\documentclass[final,preprint,3p,sort&compress]{elsarticle}
\journal{{\tt arXiv.org}}
\journal{{Journal of Computational Physics}}

\usepackage[dvips]{epsfig}
\usepackage{graphicx} 
\usepackage{pdfpages}
\usepackage{latexsym}
\usepackage{verbatim}
\usepackage{amsmath,amsthm}
\usepackage{amssymb}
\usepackage{bbm}
\usepackage{fonts}
\usepackage{enumerate}
\usepackage{placeins}
\usepackage{standalone}
\usepackage{paralist}
\usepackage{subcaption}
\definecolor{greenyellow}   {cmyk}{0.15, 0   , 0.69, 0   }
\definecolor{yellow}        {cmyk}{0   , 0   , 1   , 0   }
\definecolor{goldenrod}     {cmyk}{0   , 0.10, 0.84, 0   }
\definecolor{dandelion}     {cmyk}{0   , 0.29, 0.84, 0   }
\definecolor{apricot}       {cmyk}{0   , 0.32, 0.52, 0   }
\definecolor{peach}         {cmyk}{0   , 0.50, 0.70, 0   }
\definecolor{melon}         {cmyk}{0   , 0.46, 0.50, 0   }
\definecolor{yelloworange}  {cmyk}{0   , 0.42, 1   , 0   }
\definecolor{orange}        {cmyk}{0   , 0.61, 0.87, 0   }
\definecolor{burntorange}   {cmyk}{0   , 0.51, 1   , 0   }
\definecolor{bittersweet}   {cmyk}{0   , 0.75, 1   , 0.24}
\definecolor{redorange}     {cmyk}{0   , 0.77, 0.87, 0   }
\definecolor{mahogany}      {cmyk}{0   , 0.85, 0.87, 0.35}
\definecolor{maroon}        {cmyk}{0   , 0.87, 0.68, 0.32}
\definecolor{brickred}      {cmyk}{0   , 0.89, 0.94, 0.28}
\definecolor{red}           {cmyk}{0   , 1   , 1   , 0   }
\definecolor{orangered}     {cmyk}{0   , 1   , 0.50, 0   }
\definecolor{rubinered}     {cmyk}{0   , 1   , 0.13, 0   }
\definecolor{wildstrawberry}{cmyk}{0   , 0.96, 0.39, 0   }
\definecolor{salmon}        {cmyk}{0   , 0.53, 0.38, 0   }
\definecolor{carnationpink} {cmyk}{0   , 0.63, 0   , 0   }
\definecolor{magenta}       {cmyk}{0   , 1   , 0   , 0   }
\definecolor{violetred}     {cmyk}{0   , 0.81, 0   , 0   }
\definecolor{rhodamine}     {cmyk}{0   , 0.82, 0   , 0   }
\definecolor{mulberry}      {cmyk}{0.34, 0.90, 0   , 0.02}
\definecolor{redviolet}     {cmyk}{0.07, 0.90, 0   , 0.34}
\definecolor{fuchsia}       {cmyk}{0.47, 0.91, 0   , 0.08}
\definecolor{lavender}      {cmyk}{0   , 0.48, 0   , 0   }
\definecolor{thistle}       {cmyk}{0.12, 0.59, 0   , 0   }
\definecolor{orchid}        {cmyk}{0.32, 0.64, 0   , 0   }
\definecolor{darkorchid}    {cmyk}{0.40, 0.80, 0.20, 0   }
\definecolor{purple}        {cmyk}{0.45, 0.86, 0   , 0   }
\definecolor{plum}          {cmyk}{0.50, 1   , 0   , 0   }
\definecolor{violet}        {cmyk}{0.79, 0.88, 0   , 0   }
\definecolor{royalpurple}   {cmyk}{0.75, 0.90, 0   , 0   }
\definecolor{blueviolet}    {cmyk}{0.86, 0.91, 0   , 0.04}
\definecolor{periwinkle}    {cmyk}{0.57, 0.55, 0   , 0   }
\definecolor{cadetblue}     {cmyk}{0.62, 0.57, 0.23, 0   }
\definecolor{cornflowerblue}{cmyk}{0.65, 0.13, 0   , 0   }
\definecolor{midnightblue}  {cmyk}{0.98, 0.13, 0   , 0.43}
\definecolor{navyblue}      {cmyk}{0.94, 0.54, 0   , 0   }
\definecolor{royalblue}     {cmyk}{1   , 0.50, 0   , 0   }
\definecolor{blue}          {cmyk}{1   , 1   , 0   , 0   }
\definecolor{cerulean}      {cmyk}{0.94, 0.11, 0   , 0   }
\definecolor{cyan}          {cmyk}{1   , 0   , 0   , 0   }
\definecolor{processblue}   {cmyk}{0.96, 0   , 0   , 0   }
\definecolor{skyblue}       {cmyk}{0.62, 0   , 0.12, 0   }
\definecolor{turquoise}     {cmyk}{0.85, 0   , 0.20, 0   }
\definecolor{tealblue}      {cmyk}{0.86, 0   , 0.34, 0.02}
\definecolor{aquamarine}    {cmyk}{0.82, 0   , 0.30, 0   }
\definecolor{bluegreen}     {cmyk}{0.85, 0   , 0.33, 0   }
\definecolor{emerald}       {cmyk}{1   , 0   , 0.50, 0   }
\definecolor{junglegreen}   {cmyk}{0.99, 0   , 0.52, 0   }
\definecolor{seagreen}      {cmyk}{0.69, 0   , 0.50, 0   }
\definecolor{green}         {cmyk}{1   , 0   , 1   , 0   }
\definecolor{forestgreen}   {cmyk}{0.91, 0   , 0.88, 0.12}
\definecolor{pinegreen}     {cmyk}{0.92, 0   , 0.59, 0.25}
\definecolor{limegreen}     {cmyk}{0.50, 0   , 1   , 0   }
\definecolor{yellowgreen}   {cmyk}{0.44, 0   , 0.74, 0   }
\definecolor{springgreen}   {cmyk}{0.26, 0   , 0.76, 0   }
\definecolor{olivegreen}    {cmyk}{0.64, 0   , 0.95, 0.40}
\definecolor{rawsienna}     {cmyk}{0   , 0.72, 1   , 0.45}
\definecolor{sepia}         {cmyk}{0   , 0.83, 1   , 0.70}
\definecolor{brown}         {cmyk}{0   , 0.81, 1   , 0.60}
\definecolor{tan}           {cmyk}{0.14, 0.42, 0.56, 0   }
\definecolor{gray}          {cmyk}{0   , 0   , 0   , 0.50}
\definecolor{black}         {cmyk}{0   , 0   , 0   , 1   }
\definecolor{white}         {cmyk}{0   , 0   , 0   , 0   } 

 \usepackage[]{hyperref} 
        
\usepackage{pgfplots}
\pgfplotsset{compat=newest}

\newcommand{\externaltikz}[2]{\includegraphics{Externals/#1}}		
 \usepackage{relinput}
\newtheorem{theorem}{Theorem}[section]
\newtheorem{definition}[theorem]{Definition}
\newtheorem{remark}[theorem]{Remark}
\newtheorem{example}[theorem]{Example}
\newtheorem{assumption}[theorem]{Assumption}

\newtheorem{lemma}[theorem]{Lemma}
\newtheorem{corollary}[theorem]{Corollary}

\newcounter{tikzsubfigcounter}[figure]
\renewcommand{\thetikzsubfigcounter}{\thesection.\the\numexpr\value{figure}+1\relax\alph{tikzsubfigcounter}}

\newcounter{tikzsubfigcounterinvisible}[figure]
\renewcommand{\thetikzsubfigcounterinvisible}{\thesection.\the\numexpr\value{figure}+1\relax\alph{tikzsubfigcounterinvisible}}

\newcommand{\settikzlabel}[1]{ %
\refstepcounter{tikzsubfigcounterinvisible} \label{#1} 
}

\numberwithin{equation}{section}
\newcommand{\bdm}{\begin{displaymath}}
\newcommand{\edm}{\end{displaymath}}
\newcommand{\beq}{\begin{equation}}
\newcommand{\eeq}{\end{equation}}
\newcommand{\beqa}{\begin{eqnarray}}
\newcommand{\eeqa}{\end{eqnarray}}

\title{Kershaw closures for linear transport equations in slab geometry II: high-order realizability-preserving discontinuous-Galerkin schemes}
\author[fs]{Florian Schneider}
\address[fs]{Fachbereich Mathematik, TU Kaiserslautern, Erwin-Schr\"odinger-Str., 67663 Kaiserslautern, Germany, {\tt schneider@mathematik.uni-kl.de}}

\date{}

\usepgfplotslibrary{groupplots}
\usepackage{booktabs} % Required for nicer horizontal rules in tables
\newlength{\figureheight}
\newlength{\figurewidth}
\usetikzlibrary{arrows,shapes,positioning}
\usetikzlibrary{decorations.markings}
\tikzstyle arrowstyle=[scale=1]
\tikzstyle directed=[postaction={decorate,decoration={markings,
		mark=at position .65 with {\arrow[arrowstyle]{stealth}}}}]
\tikzstyle reverse directed=[postaction={decorate,decoration={markings,
		mark=at position .65 with {\arrowreversed[arrowstyle]{stealth};}}}]

\newcommand{\secref}[1]{Section~\ref{#1}}
\newcommand{\thmref}[1]{Theorem~\ref{#1}}

\newcommand{\lemref}[1]{Lemma~\ref{#1}}

\renewcommand{\corref}[1]{Corollary~\ref{#1}}
\newcommand{\figref}[1]{Figure~\ref{#1}}
\newcommand{\tabref}[1]{Table~\ref{#1}}

\newcommand{\abs}[1]{\ensuremath{\left| #1 \right|}}

\newcommand{\N}{\mathbb{N}} 

\newcommand{\R}{\mathbb{R}}
\newcommand{\Rpos}{\R_{\geq 0}}
\newcommand{\scattering}{\ensuremath{\sigma_s}}
\newcommand{\absorption}{\ensuremath{\sigma_a}}
\newcommand{\crosssection}{\ensuremath{\sigma_t}}
\newcommand{\maxcrosssection}{\ensuremath{\crosssection^{\max}}}

\newcommand{\source}{\ensuremath{Q}}

\newcommand{\ceil}[1]{\ensuremath{\left\lceil #1 \right\rceil}}

\newcommand{\sphere}[1][2]{\ensuremath{\mathcal{S}^{#1}}}

\newcommand{\distribution}[1][ ]{\ensuremath{\psi_{#1}}}
\newcommand{\distributiontzero}{\ensuremath{\distribution[\timevar=0]}}
\newcommand{\distributionboundary}{\ensuremath{\distribution[b]}}
\newcommand{\distributionvacuum}{\ensuremath{\distribution[\text{vac}]}}
\newcommand{\ansatz}[1][ ]{\ensuremath{\hat{\psi}_{#1}}}

\newcommand{\momentorder}{\ensuremath{N}}
\newcommand{\momentnumber}{\ensuremath{n}}
\newcommand{\basis}[1][ ]{{\ensuremath{\bb_{#1}}}} %Basis
\newcommand{\basisind}{\ensuremath{i}} %Basis
\newcommand{\basiscomp}[1][\basisind]{\ensuremath{b_{#1}}} %Basis
\newcommand{\normalizedbasis}[1][ ]{{\ensuremath{\widehat{\bb}_{#1}}}} %Basis
\newcommand{\fmbasis}[1][]{\basis[#1]} %Basis
\newcommand{\moments}[1][ ]{\ensuremath{\bu_{#1}}} %moment vector
\newcommand{\momentcomp}[1]{\ensuremath{u_{#1}}} %moment vector
\newcommand{\momentstzero}{\ensuremath{\moments[\timevar=0]}}

\newcommand{\isotropicmoment}{\moments[\text{iso}]}
\newcommand{\normalizedmoments}[1][ ]{\ensuremath{\bsphi_{#1}}} %moment vector
\newcommand{\normalizedmomentcomp}[1]{\ensuremath{\phi_{#1}}} %moment vector
\newcommand{\normalizedisotropicmoment}{\normalizedmoments[\text{iso}]}
\newcommand{\convexscalar}{\ensuremath{\zeta}} %scalar indicating a convex combination
\newcommand{\flow}{\ensuremath{f_{\text{low}}}} %eddington factor
\newcommand{\fup}{\ensuremath{f_{\text{up}}}} %eddington factor

\newcommand{\SC}{\ensuremath{\Omega}} %moment vector
\newcommand{\SCheight}{\ensuremath{\mu}} %moment vector
\newcommand{\Domain}{\ensuremath{X}} %moment vector
\newcommand{\spatialVariable}{\ensuremath{\bx}} %moment vector
\newcommand{\timeint}{\ensuremath{T}} %time interval
\newcommand{\tf}{\ensuremath{t_f}} %final time
\newcommand{\timevar}{\ensuremath{t}} %final time

\newcommand{\ints}[1]{\ensuremath{\left<#1\right>}}

\newcommand{\collisionop}{\ensuremath{\cC}}
\newcommand{\collision}[1]{\ensuremath{\collisionop\left(#1\right)}}
\newcommand{\lincollisionop}{\ensuremath{\cI}}
\newcommand{\lincollision}[1]{\ensuremath{\lincollisionop\left(#1\right)}}
\newcommand{\collisionkernel}{\ensuremath{K}}
\newcommand{\collisionrealizablepart}[1][ ]{\ensuremath{\widetilde{\moments[#1]}}}

\newcommand{\dirac}{\ensuremath{\delta}}

\newcommand{\Lp}[1]{\ensuremath{L_{#1}}}
\newcommand{\RD}[2]{\ensuremath{\mathcal{R}_{#1}^{#2}}}

\newcommand{\PN}[1][\momentorder]{\ensuremath{\text{P}_{#1}}}
\newcommand{\MN}[1][\momentorder]{\ensuremath{\text{M}_{#1}}}
\newcommand{\KN}[1][\momentorder]{\ensuremath{\text{K}_{#1}}}

\newcommand{\Flux}{\ensuremath{\bF}}
\newcommand{\Source}{\ensuremath{\bs}}

\newcommand{\entropy}{\ensuremath{\eta}} %Entropy function
\def\quand{\quad \mbox{and} \quad}

\newcommand{\z}{\ensuremath{z}}

\newcommand{\dz}{\partial_{\z}}
\newcommand{\dt}{\partial_\timevar}

\newcommand{\hankelA}{A}
\newcommand{\hankelB}{B}
\newcommand{\hankelC}{C}

\newcommand{\hankelb}{\bsbeta} %Eventuell austauschen.........................................................

\newcommand{\hankelhalfind}{k}

\newcommand{\pseudoinv}[1]{\ensuremath{#1^\dag}}

\newcommand{\dtstepsize}{\ensuremath{\Delta \timevar}}

\newcommand{\zL}{\ensuremath{\z_{L}}}
\newcommand{\zR}{\ensuremath{\z_{R}}}
\newcommand{\dzstepsize}{\ensuremath{\Delta\z}}
\newcommand{\ncells}{\ensuremath{n_{\z}}}

\newcommand{\timeind}{\ensuremath{\kappa}}
\newcommand{\timepoint}[1]{\ensuremath{\timevar_{#1}}}
\newcommand{\cellind}{\ensuremath{j}}

\newcommand{\cell}[1]{\ensuremath{I_{#1}}}
\newcommand{\referencecell}{\ensuremath{\hat{I}}}
\newcommand{\referencevar}{\ensuremath{\hat{\z}}}
\newcommand{\polybasisind}{\ensuremath{i}}

\newcommand{\spatialorder}{\ensuremath{k}}
\newcommand{\spatialordersource}{\ensuremath{k_{s}}}
\newcommand{\polydegree}{\ensuremath{\spatialorder-1}}
\newcommand{\cellmean}[2][\cellind]{\ensuremath{\overline{#2}_{#1}}}

\newcommand{\limitervariable}{\ensuremath{\theta}}

\newcommand{\momentsprojected}{\ensuremath{\moments[h]}}
\newcommand{\momentcompprojected}[1]{\ensuremath{\momentcomp{h,#1}}}
\newcommand{\momentspv}[1]{\ensuremath{\moments[#1]}}

\newcommand{\momentcompcellmean}[1]{\ensuremath{\cellmean[#1]{\momentcomp{}}}}
\newcommand{\momentcomplimitedpv}[2][\limitervariable]{\momentcomp{#2}^{#1}}
\newcommand{\momentscollection}[1][ ]{\ensuremath{\widehat{\moments}_{h}^{#1}}}
\newcommand{\momentscellmean}[1]{\ensuremath{\cellmean[#1]{\moments}}}
\newcommand{\momentscellmeantime}[2]{\ensuremath{\cellmean[#1]{\moments}^{\left(#2\right)}}}
\newcommand{\momentslocal}[1]{\moments[#1]}
\newcommand{\momentslocallimited}[2][\limitervariable]{\moments[#2]^{#1}}
\newcommand{\momentslimitedpv}[2][\limitervariable]{\moments[#2]^{#1}}
\newcommand{\momentspolynomialcoefficients}[2]{\widehat{\moments}_{#1}^{#2}}
\newcommand{\momentspolynomialmatrix}[1]{\widehat{\moments}_{#1}}
\newcommand{\Testfunction}[1][ ]{\ensuremath{v_{#1}}}
\newcommand{\FiniteElementSpace}[1]{\ensuremath{V_h^{#1}}}

\newcommand{\SpaceOfPolynomials}[1]{\ensuremath{P^{#1}}}
\newcommand{\numericalFlux}{\ensuremath{\widehat{\bF}}}
\newcommand{\viscosityconstant}{\ensuremath{C}}
\newcommand{\DifferentialOperatorLocal}{\ensuremath{\tilde{L}_h}}

\newcommand{\LpError}[2][h]{\ensuremath{E_{#1}^{#2}}}

\newcommand{\minmod}[1]{\ensuremath{m\left(#1\right)}}

\newcommand{\slopelimiter}{\ensuremath{\Lambda^{\text{scalar}}}}

\newcommand{\TVBconstant}{\ensuremath{M}}

\def\analyticalsolution{\distribution[a]}

\newcommand{\analyticalmomentcomp}[1]{\ensuremath{\momentcomp{a,#1}}}
\newcommand{\MFSconstorder}{\ensuremath{\nu}}
\newcommand{\MFSp}{\ensuremath{p}}
\newcommand{\MFSgamma}{\ensuremath{\gamma}}

\newcommand{\spatialQuadrature}{\ensuremath{\Upsilon}}

\newcommand{\spatialQuadratureIndex}{\ensuremath{\upsilon}}
\newcommand{\spatialQuadratureNumber}{\ensuremath{{n_\spatialQuadrature}}}

\newcommand{\spatialQuadratureWeightsRef}[1]{\ensuremath{\hat{w}_{#1}}}
\newcommand{\spatialQuadratureNodes}[2]{\ensuremath{\z_{#1,#2}}}

\pgfplotscreateplotcyclelist{color parula}{% 
	royalblue!70,every mark/.append style={solid,line width = 0pt,fill=royalblue!60!black},mark=ball\\%
	goldenrod!50!yelloworange,every mark/.append style={solid,fill=goldenrod!30!black},mark=*\\%
	green,every mark/.append style={black,fill=yellowgreen},mark=diamond*\\%
	bluegreen,every mark/.append style={fill=black},mark=triangle*\\%	
	royalblue!70,densely dashed,every mark/.append style={solid,line width = 0pt,fill=royalblue!60!black},mark=ball\\%
	goldenrod!50!yelloworange,densely dashed,every mark/.append style={solid,black,fill=goldenrod},mark=diamond*\\%
	yellowgreen,densely dashed,every mark/.append style={solid,fill=yellowgreen!60!royalblue},mark=square*, mark size= 1pt\\%
	bluegreen,densely dashed,mark size = 1.5pt,every mark/.append style={solid,fill=white},mark=otimes*\\%,densely dashed	
	royalblue,densely dashed,mark size = 1.5pt,every mark/.append style={solid,fill=royalblue!30!black},mark=pentagon*\\%,densely dashed
	goldenrod!40!yelloworange,every mark/.append style={solid,fill=goldenrod!30!black},mark=|\\%		
	}

\begin{document}

\begin{abstract}
This paper provides a generalization of the realizability-preserving discontinuous-Galerkin scheme given in \cite{Schneider2015a} to general full-moment models that can be closed analytically. It is applied to the class of Kershaw closures, which are able to provide a cheap closure of the moment problem. This results in an efficient algorithm for the underlying linear transport equation. The efficiency of high-order methods is demonstrated using numerical convergence tests and non-smooth benchmark problems.
\end{abstract}
\maketitle

\noindent

% {\bf Key words.}

\section{Introduction}
Moment closures are a class of spectral methods used in the context of
kinetic transport equations.
An infinite set of moment equations is defined by taking velocity-
or phase-space averages with respect to some basis of the velocity space. A
reduced description of the kinetic density is then
achieved by truncating this hierarchy of equations at some finite order.
The remaining equations however inevitably require information from the
equations which were removed.
The specification of this information, the so-called moment closure problem, distinguishes different moment
methods.
In the context of linear radiative transport, the standard spectral method
is commonly referred to as the $\PN$ closure \cite{Lewis-Miller-1984},
where $\momentorder$ is the degree of the highest-order moments in the model.
The $\PN$ method is powerful and simple to implement, but does not
take into account the fact that the original function to be
approximated, the kinetic density, must be non-negative.
Thus $\PN$ solutions can contain negative values for the local
densities of particles, rendering the solution physically meaningless.

Entropy-based moment closures, referred to as $\MN$ models in the
context of radiative transport \cite{Min78,DubFeu99},
have all the properties one would desire in a moment method, namely
positivity of the underlying kinetic density,%
\footnote{
Positivity is actually not gained for every entropy-based moment closure
but is indeed a property of those models derived from important, physically
relevant entropies.
}
hyperbolicity of the closed system of equations,
and entropy dissipation \cite{Lev96}.
Practical implementation of these models has been traditionally considered
too expensive because they require the numerical solution of an
optimization problem at every point on the space-time grid, but recently
there has been renewed interest in the models due to their inherent
parallelizability \cite{Hauck2010}.
However, while their parallelizability goes a long way in making $\MN$
models computationally competitive, in order to make these methods truly
competitive with more basic discretizations, the gains in efficiency that
come from higher-order methods ({in space and time}) will likely be necessary.
Here the issue of realizability becomes a stumbling block.

The property of positivity implies that the system of moment
equations only evolves on the set of so-called realizable moments.
Realizable moments are simply those moments associated with positive
densities, and the set of these moments forms a convex cone which is a
strict subset of all moment vectors.
This property, while indeed desirable since it is consistent with the
original kinetic density, can cause problems for numerical methods.
Standard high-order numerical solutions {(in space and time)} to the Euler equations, which indeed are an
entropy-based moment closure, have been observed to have negative
local densities and pressures \cite{Zhang2010}. {Similar effects have been reported in the context of elastic flow \cite{Schar1996}.}
This is exactly loss of realizability.

A recently popular high-order method for hyperbolic systems is the
Runge-Kutta discontinuous Galerkin (RKDG) method
\cite{Cockburn1989,Cockburn1989a,Cockburn1991}.
An RKDG method for moment closures can handle the loss of realizability
through the use of a realizability (or ``positivity-preserving'') limiter
\cite{Zhang2010}, but so far these have been implemented for low-order
moment systems (that is $\momentorder = 1$ or $2$) \cite{Olbrant2012}
because here one can rely on the
simplicity of the structure of the realizable set for low-order moments.
{For moments of large order $\momentorder$}, the realizable set has complex nonlinear
boundaries: when the velocity domain is one-dimensional, the realizable
set is characterized by the positive-definiteness of Hankel matrices
\cite{Shohat1943,Curto1991}; in higher dimensions, the realizable
set is not well understood.
In \cite{Schneider2015a}, using that a quadrature-based
approximation of the realizable set is a convex polytope \cite{Alldredge2014}, the realizability
limiters of \cite{Zhang2010,Olbrant2012} has been generalized for moment systems of (in
principle) arbitrary order.

To avoid the expensive minimum-entropy ansatz a new hierarchy of full-moment models has been derived in \cite{Schneider2015}, the class of Kershaw closures, based on the findings in \cite{Ker76}. It provides a reasonably simple closure relation, closely related to minimum-entropy models while being cheap to evaluate. This paper aims at generalizing the scheme given in \cite{Schneider2015a} to this class of models for (in principle) arbitrary moment order $\momentorder$.

This paper is organized as follows. First, the transport equation and its moment approximations are given. Then, the available realizability theory is shortly reviewed, followed by a brief summary of the class of Kershaw closures. The discontinuous-Galerkin scheme is given with the necessary extensions to obtain a realizability-preserving scheme. Numerical convergence of this scheme up to seventh order against an analytical solution is shown and the Kershaw closures are submitted to a set of benchmark tests investigating the effect of high-order space-time approximations. Finally, conclusions and an outlook on future work is given.

\section{Modelling}
In slab geometry, the transport equation under consideration has the form 
\begin{align}
\label{eq:TransportEquation1D}
\dt\distribution+\SCheight\dz\distribution + \absorption\distribution = \scattering\collision{\distribution}+\source, \qquad \timevar\in\timeint,\z\in\Domain,\SCheight\in[-1,1].
\end{align}
The physical parameters are the absorption and scattering coefficient $\absorption,\scattering:\timeint\times\Domain\to\Rpos$, respectively, and the emitting source $\source:\timeint\times\Domain\times[-1,1]\to\Rpos$. Furthermore, $\SCheight\in[-1,1]$, and $\distribution = \distribution(\timevar,\z,\SCheight)$. 

The shorthand notation $\ints{\cdot} = \int\limits_{-1}^1\cdot~d\SCheight$ denotes integration over $[-1,1]$.

\begin{assumption}
\label{ass:CollisionOperator}
Following \cite{Levermore1996}, the collision operator $\collisionop$ is assumed to have the following properties.
\begin{enumerate}
\begin{subequations}
\label{eq:CollisionProperty}
\item Mass conservation
\begin{align}
\label{eq:CollisionPropertyMass}
\ints{\collision{\distribution}}=0.
\end{align}
\item Local entropy dissipation
\begin{align}
\label{eq:CollisionPropertyLocalDissipation}
\ints{\entropy'(\distribution)\collision{\distribution}}\leq 0,
\end{align}
where $\entropy$ denotes a strictly convex, twice differentiable entropy.
\item Constants in the kernel: 
\begin{align}
\label{eq:ConstantKernel}
\collision{c} = 0 \qquad \text{for every } c\in\R.
\end{align}
\end{subequations}
\end{enumerate}
\end{assumption}

A typical example for $\collisionop$ is the linear integral operator
\begin{equation}
 \lincollision{\distribution} =  \int\limits_{-1}^1 \collisionkernel(\SCheight, \SCheight^\prime)
  \distribution(\timevar, \z, \SCheight^\prime)~d\SCheight^\prime 
  - \distribution(\timevar, \z, \SCheight),
\label{eq:collisionOperatorLin1D}
\end{equation}
where $\collisionkernel$ is non-negative, symmetric in both arguments and normalized to $\int\limits_{-1}^1 \collisionkernel(\SCheight, \SCheight^\prime)~d\SCheight^\prime=1$.
In this paper the special case of the BGK-type isotropic-scattering operator with $\collisionkernel\equiv \frac12$ is used for the simulations.

\eqref{eq:TransportEquation1D} is supplemented by initial and boundary conditions:
\begin{subequations}
\begin{align}
\distribution(0,\z,\SCheight) &= \distributiontzero(\z,\SCheight) &\text{for } \z\in\Domain = (\zL,\zR), \SCheight\in[-1,1], \label{eq:TransportEquation1DIC}\\
\distribution(\timevar,\zL,\SCheight) &= \distributionboundary(\timevar,\zL,\SCheight) &\text{for } \timevar\in\timeint, \SCheight>0,  \label{eq:TransportEquation1DBCa}\\
\distribution(\timevar,\zR,\SCheight) &= \distributionboundary(\timevar,\zR,\SCheight) &\text{for } \timevar\in\timeint, \SCheight<0. \label{eq:TransportEquation1DBCb}
\end{align}
\end{subequations}

%Integrals over the ``half-sphere'' intervals $[-1,0]$ and $[0,1]$ are denoted by 
%\begin{align*}
%\intm{\cdot} = \int\limits_{-1}^{0}\cdot~d\SCheight,
%&&\intp{\cdot} = \int\limits_{0}^{1}\cdot~d\SCheight.
%\end{align*}
\section{Moment models and realizability}
In general, solving equation \eqref{eq:TransportEquation1D} is very expensive in two and three dimensions due to the high dimensionality of the state space. 

For this reason it is convenient to use some type of spectral or Galerkin method to transform the high-dimensional equation into a system of lower-dimensional equations. Typically, one chooses to reduce the dimensionality by representing the angular dependence of $\distribution$ in terms of some basis $\basis$.
\begin{definition}
The vector of functions $\basis:[-1,1]\to\R^{\momentorder+1}$ consisting of $\momentorder+1$ basis functions $\basiscomp[\basisind]$, $\basisind=0,\ldots\momentorder$ of maximal \emph{order} $\momentorder$ is called an \emph{angular basis}.

The so-called \emph{moments} of a given distribution function $\distribution$ with respect to $\basis$ are then defined by
\begin{align}
\label{eq:moments}
\moments =\ints{{\basis}\distribution} = \left(\momentcomp{0},\ldots,\momentcomp{\momentorder}\right)^T,
\end{align}
where the integration is performed componentwise.\\

Assuming for simplicity $\basiscomp[0]\equiv 1$, the quantity $\momentcomp{0} = \ints{\basiscomp[0]\distribution}=\ints{\distribution}$ is called \emph{local particle density}. 
Furthermore, \emph{normalized moments} $\normalizedmoments = \left(\normalizedmomentcomp{1},\ldots,\normalizedmomentcomp{\momentorder}\right)\in\R^{\momentorder}$ are defined as 
\begin{align}
\label{eq:NormalizedMoments}
\normalizedmoments = \cfrac{\ints{\normalizedbasis\distribution}}{\ints{\distribution}}~,
\end{align}
where $\normalizedbasis = \left(\basiscomp[1],\ldots,\basiscomp[\momentorder]\right)^T$ is the remainder of the basis $\basis$.
\end{definition}
To obtain a set of equations for $\moments$, \eqref{eq:TransportEquation1D} has to be multiplied through by $\basis$ and integrated over $[-1,1]$, giving
\begin{align*}
\ints{\basis\dt\distribution}+\ints{\basis\dz\SCheight\distribution} + \ints{\basis\absorption\distribution} = \scattering\ints{\basis\collision{\distribution}}+\ints{\basis\source}.
\end{align*}
Collecting known terms, and interchanging integrals and differentiation where possible, the moment system has the form
\begin{align}
\label{eq:MomentSystemUnclosed1D}
\dt\moments+\dz\ints{\SCheight \basis\ansatz[\moments]} + \absorption\moments = \scattering\ints{\basis\collision{\ansatz[\moments]}}+\ints{\basis\source}.
\end{align}
The solution of \eqref{eq:MomentSystemUnclosed1D} is equivalent to the one of \eqref{eq:TransportEquation1D} if $\basis$ is a basis of $\Lp{2}(\sphere,\R)$. 

Since it is impractical to work with an infinite-dimensional system, only a finite number of $\momentorder+1<\infty$ basis functions $\basis$ of order $\momentorder$ can be considered. Unfortunately, there always exists an index $\basisind\in\{0,\dots,\momentorder\}$ such that the components of $\basiscomp\cdot\SCheight$ are not in the linear span of $\basis[\momentorder]$. Therefore, the flux term cannot be expressed in terms of $\moments[\momentorder]$ without additional information. Furthermore, the same might be true for the projection of the scattering operator onto the moment-space given by $\ints{\basis\collision{\distribution}}$. This is the so-called \emph{closure problem}. One usually prescribes some \emph{ansatz} distribution $\ansatz[\moments](\timevar,\spatialVariable,\SC):=\ansatz(\moments(\timevar,\spatialVariable),\basis(\SC))$ to calculate the unknown quantities in \eqref{eq:MomentSystemUnclosed1D}. Note that the dependence on the angular basis in the short-hand notation $\ansatz[\moments]$ is neglected for notational simplicity.\\

In this paper, the \emph{full-moment monomial basis} $\basiscomp = \SCheight^\basisind$ is considered. However, it is in principle possible to extend the derived concepts to other bases like half \cite{DubKla02,DubFraKlaTho03} or mixed moments \cite{Frank07,Schneider2014}.

The rest of this section is a brief summary of the corresponding parts in \cite{Schneider2015}. All details and further discussions can be found therein.

\subsection{Realizability}
Since the underlying kinetic density to be approximated is
non-negative, a 
moment vector only makes sense physically if it can be associated with a 
non-negative distribution function. In this case the moment vector is called 
\emph{realizable}.

\begin{definition}
\label{def:RealizableSet}
The \emph{realizable set} $\RD{\basis}{}$\index{Realizability@\textbf{Realizability}!Realizable set $\RD{\basis}{}$} is 
$$
\RD{\basis}{} = \left\{\moments~:~\exists \distribution(\SCheight)\ge 0,\, \ints{\distribution} > 0,
 \text{ such that } \moments =\ints{\basis\distribution} \right\}.
$$
If $\moments\in\RD{\basis}{}$, then $\moments$ is called \emph{realizable}.
Any $\distribution$ such that $\moments =\ints{\basis \distribution}$ is called a \emph{representing 
density}. If $\distribution$ is additionally a linear combination of Dirac deltas \cite{Hassani2009,Mathematics2011,Kuo2006}, it is called \emph{atomic} \cite{Curto1991}.
\end{definition}

{\begin{definition}
\mbox{ }
Let $\hankelA,\hankelB\in\R^{\momentnumber\times \momentnumber}$ be Hermitian matrices. The partial ordering $"\geq"$ on such matrices is defined by $\hankelA\geq \hankelB$ if and only if $\hankelA-\hankelB$ is positive semi-definite. In particular $\hankelA\geq 0$ denotes that $\hankelA$ is positive semi-definite.
\end{definition}}

For the full-moment basis the question of finding practical characterizations of the realizable set $\RD{\basis}{}$ has been completely solved in \cite{Curto1991}. See \cite{Schneider2015} for more details.
The following characterizations of the above realizable set holds.
\begin{lemma}
\label{thm:FullMomentRealizability}
Define the \emph{Hankel matrices}
$$
\hankelA(\hankelhalfind):=\left(\momentcomp{i+j}\right)_{i,j=0}^\hankelhalfind, \quad \hankelB(\hankelhalfind):=\left(\momentcomp{i+j+1}\right)_{i,j=0}^\hankelhalfind, \quad \hankelC(\hankelhalfind):=\left(\momentcomp{i+j}\right)_{i,j=1}^\hankelhalfind.
$$
Then the realizable set satisfies
\begin{align*}
\RD{\fmbasis}{} = \begin{cases}
\left\{\moments\in\R^{\momentorder+1}~|~ \hankelA(\hankelhalfind)\geq \hankelB(\hankelhalfind),~\hankelA(\hankelhalfind)\geq - \hankelB(\hankelhalfind)\right\} & \text{ if } $\momentorder=2\hankelhalfind+1$,\\
\left\{\moments\in\R^{\momentorder+1}~|~ \hankelA(\hankelhalfind)\geq 0, \hankelA(\hankelhalfind-1)\geq \hankelC(\hankelhalfind)\right\} & \text{ if } $\momentorder=2\hankelhalfind$.\\
\end{cases}
\end{align*}
\end{lemma}

Due to the structure of the used Hankel matrices (the highest moment $\momentcomp{\momentorder}$ always appears exactly once in the entries of the matrices) it is always possible to rearrange the conditions involving this highest moment in \thmref{thm:FullMomentRealizability} in such a way that 
\begin{align*}
%\label{eq:UpperLowerBoundsFullMoments}
\fup(\momentcomp{0},\ldots,\momentcomp{\momentorder-1})\geq \momentcomp{\momentorder} \geq \flow(\momentcomp{0},\ldots,\momentcomp{\momentorder-1})
\end{align*}
for functions $\fup$ and $\flow$. Whenever $\moments$ is realizable and $\momentcomp{\momentorder} = \flow(\momentcomp{0},\ldots,\momentcomp{\momentorder-1})$, $\moments$ is said to be on the \emph{lower $\momentorder^{\text{th}}$-order realizability boundary}. Similarly, if $\momentcomp{\momentorder} = \fup(\momentcomp{0},\ldots,\momentcomp{\momentorder-1})$, $\moments$ is said to be on the \emph{upper $\momentorder^{\text{th}}$-order realizability boundary}.

The functions $\fup$ and $\flow$ can be specified using the pseudoinverses of combinations of Hankel matrices.
To simplify notation later, the following corollary is written in terms of $\momentcomp{\momentorder+1}$ instead of $\momentcomp{\momentorder}$.
\begin{corollary}
\label{cor:FupFlow}
The functions $\fup$ and $\flow$ satisfying 
\begin{align}
\label{eq:UpperLowerBoundsFullMoments}
\fup(\momentcomp{0},\ldots,\momentcomp{\momentorder})\geq \momentcomp{\momentorder+1} \geq \flow(\momentcomp{0},\ldots,\momentcomp{\momentorder})
\end{align}
are given by
\begin{align*}
\fup(\momentcomp{0},\ldots,\momentcomp{\momentorder}) &= 
\begin{cases}
\momentcomp{\momentorder-1}-\hankelb_-^T\pseudoinv{\left(\hankelA(\hankelhalfind-1)-\hankelC(\hankelhalfind-1)\right)}\hankelb_- & \text{ if } \momentorder = 2\hankelhalfind+1\\
\momentcomp{\momentorder}-\hankelb_-^T\pseudoinv{\left(\hankelA(\hankelhalfind-1)-\hankelB(\hankelhalfind-1)\right)}\hankelb_- & \text{ if } \momentorder = 2\hankelhalfind
\end{cases}\\
\flow(\momentcomp{0},\ldots,\momentcomp{\momentorder}) &= 
\begin{cases}
\hankelb_+^T\pseudoinv{\hankelA}(\hankelhalfind)\hankelb_+& \text{ if } \momentorder = 2\hankelhalfind+1\\
-\momentcomp{\momentorder}+\hankelb_+^T\pseudoinv{\left(\hankelA(\hankelhalfind-1)+\hankelB(\hankelhalfind-1)\right)}\hankelb_+~ & \text{ if } \momentorder = 2\hankelhalfind
\end{cases}
\end{align*}
where in the odd case
\begin{align*}
\hankelb_- = \left(\momentcomp{\hankelhalfind}-\momentcomp{\hankelhalfind+2},\ldots,\momentcomp{\momentorder-2}-\momentcomp{\momentorder}\right)^T,\qquad \hankelb_+ = \left(\momentcomp{\hankelhalfind+1},\ldots,\momentcomp{\momentorder}\right)^T
\end{align*}
and in the even case
\begin{align*}
\hankelb_\mp = \left(\momentcomp{\hankelhalfind}\mp\momentcomp{\hankelhalfind+1},\ldots,\momentcomp{\momentorder-1}\mp\momentcomp{\momentorder}\right)^T.
\end{align*}
\end{corollary}

\begin{remark}
By convention, $\hankelb_-^T\pseudoinv{\left(\hankelA(\hankelhalfind-1)-\hankelC(\hankelhalfind-1)\right)}\hankelb_- = 0$ if $\momentorder = 1$.
\end{remark}
\subsection{Kershaw closures}
With the previous realizability theory it is now possible to develop the closure strategy which is called \emph{Kershaw} closure. This class of moment models is defined by convexly combining upper and lower moments of order $\momentorder+1$ in such a way that the isotropic point is correctly reproduced.

\begin{corollary}
\mbox{ }\\
The Kershaw closure $\KN$ of order $\momentorder$ is given by
\begin{align}
\label{eq:KnAnsatz}
\normalizedmomentcomp{\momentorder+1}(\normalizedmoments) = \convexscalar \flow(\normalizedmoments)+(1-\convexscalar) \fup(\normalizedmoments),
\end{align}
where the interpolation constant
\begin{align}
\label{eq:KnAnsatzScalar}
\convexscalar = \cfrac{\frac12\ints{\SCheight^{\momentorder+1}}-\fup(\normalizedisotropicmoment)}{\flow(\normalizedisotropicmoment)-\fup(\normalizedisotropicmoment)} = 
\begin{cases}
\frac{\hankelhalfind+2}{2\hankelhalfind+3} & \text{ if } \momentorder = 2\hankelhalfind+1\\
\frac12 & \text{ if } \momentorder = 2\hankelhalfind
\end{cases}
\end{align}
is defined via the functions $\fup$ and $\flow$ as given in \corref{cor:FupFlow} and $\normalizedisotropicmoment = \frac{\ints{\basis}}{2}$.
\end{corollary}

For convenience, \eqref{eq:MomentSystemUnclosed1D} using the Kershaw closure can be written in the form of a usual first-order system of balance laws
\begin{align}
\label{eq:GeneralHyperbolicSystem1D}
\dt\moments+\dz\Flux_3\left(\moments\right) = \Source\left(\moments\right),
\end{align}
where 
\begin{subequations}
\label{eq:FluxDefinitions}
\begin{align}
\Flux\left(\moments\right) &= \left(\momentcomp{1},\ldots,\momentcomp{\momentorder+1}\right)\in\R^{\momentorder+1},\\
\Source\left(\moments\right) &= \scattering\left(\frac12 \normalizedisotropicmoment\momentcomp{0}-\moments\right)+\ints{\basis\source}-\absorption\moments.
\end{align}
\end{subequations}
\section{Realizability-preserving discontinuous-Galerkin scheme}
Recent numerical experiments have shown that high-order schemes {(in space and time)} outperform highly-resolved first-order methods, comparing degrees of freedom and running time versus approximation quality. This has been investigated in the case of minimum-entropy moment models in \cite{Schneider2015a,Schneider2015b} for two different types of schemes. The most challenging part is to preserve realizability during the simulation since otherwise the closure cannot be evaluated. Unfortunately, higher-order schemes typically cannot guarantee this property on their own, as
has been observed in the context of the compressible Euler equations (which are
indeed in the hierarchy of minimum-entropy models) in \cite{Zhang2010} and for the $\MN[1]$ model in \cite{Olbrant2012}.

Due to the lack of smoothness in the underlying distribution of the Kershaw models (since it is atomic) the application of the high-order kinetic scheme presented in \cite{Schneider2015b} is not obvious. This has been observed before in \cite{Vikas2011} for quadrature-based moment methods. Therefore this paper focuses on the \textbf{discontinuous-Galerkin scheme} presented in \cite{Schneider2015a}. While there only quadrature-based minimum-entropy models have been investigated, the following sections will show how to generalize the scheme and its realizability limiter to the general case of full-moment models.\\

In the following, the spatial domain $\Domain = (\zL, \zR)$ is divided into (for notational simplicity) $\ncells$ (equidistant) cells $\cell{\cellind} = (\z_{\cellind-\frac12}, \z_{\cellind+\frac12})$, where the cell interfaces are given by $\z_{\cellind\pm\frac12} = \z_\cellind \pm 
\frac{\dzstepsize}{2}$ for cell centres $\z_\cellind = \zL + (\cellind - \frac12)\dzstepsize$, and
$\dzstepsize = \frac{\zR - \zL}{\ncells}$. 

Furthermore, $\SpaceOfPolynomials{\spatialorder}(\cell{\cellind})$ is the set of polynomials of degree at most $\spatialorder$ on the interval $\cell{\cellind}$, and 
\begin{equation}
\FiniteElementSpace{\spatialorder} = \{\Testfunction \in \Lp{1}(\Domain): \Testfunction|_{\cell{\cellind}} \in \SpaceOfPolynomials{\polydegree}(\cell{\cellind}) \text{ for } 
 \cellind \in \{ 1, \ldots , \ncells \} \}
\label{eq:dg-space}
\end{equation}
is the finite-element space of piecewise polynomials of degree $\polydegree$.\\

The discontinuous-Galerkin method for the general hyperbolic system \eqref{eq:GeneralHyperbolicSystem1D}, as outlined in \cite{Cockburn1989,Cockburn1989a,Cockburn1991}, can be briefly described as follows.

For each $\timevar\in\timeint$, seek an approximate solution $\momentsprojected(\timevar, \z)$ whose components live in the finite-element space $\FiniteElementSpace{\spatialorder}$ as defined in \eqref{eq:dg-space}.

Then follow the Galerkin approach: replace $\moments$ in \eqref{eq:GeneralHyperbolicSystem1D} by a 
solution of the form $\momentsprojected \in \FiniteElementSpace{\spatialorder}$, multiply the resulting equation by basis functions $\Testfunction[h]$ of $\FiniteElementSpace{\spatialorder}$ and integrate over cell $\cell{\cellind}$ to obtain
\begin{subequations}
\label{eq:dweakform1}
\begin{align}
 \dt \int_{\cell{\cellind}} \momentsprojected(\timevar, \z)\Testfunction[h](\z)~d\z
 &+ \Flux_3(\momentsprojected(\timevar, \z_{\cellind+\frac12}^-)) \Testfunction[h](\z_{\cellind+\frac12}^-)
  - \Flux_3(\momentsprojected(\timevar, \z_{\cellind-\frac12}^+)) \Testfunction[h](\z_{\cellind-\frac12}^+) \nonumber \\
 &-\int_{\cell{\cellind}} \Flux_3(\momentsprojected(\timevar, \z)) \dz \Testfunction[h](\z)~d\z 
 = \int_{\cell{\cellind}} \Source(\momentsprojected(\timevar, \z))\Testfunction[h](\z)~d\z  \label{eq:dweakform1a},\\
 \int_{\cell{\cellind}} \momentsprojected(0, \z)\Testfunction[h](\z)~d\z &= \int_{\cell{\cellind}} \momentstzero(\z) 
\Testfunction[h](\z)~d\z,
 \label{eq:dweakform1b}
\end{align}
\end{subequations}
where $\z_{\cellind \pm \frac12}^-$ and $\z_{\cellind \pm \frac12}^+$ again denote the limits 
from left and right, respectively, and $\momentstzero = \ints{\basis\distributiontzero}$ is the projection of the initial 
distribution to the moment space.
In order to approximately solve the Riemann problem at the cell-interfaces,
the fluxes $\Flux_3(\momentsprojected(\timevar, \z_{\cellind + \frac12}^\pm))$ at the points of discontinuity are both replaced by a 
numerical flux $\numericalFlux(\momentsprojected(\timevar, \z_{\cellind+\frac12}^-), \momentsprojected(\timevar, \z_{\cellind+\frac12}^+))$, thus 
coupling the elements with their neighbours \cite{Toro2009}. Several well-known examples for such a numerical flux $\numericalFlux$ exist in literature. The simplest example is the global Lax-Friedrichs flux
\begin{align}
\label{eq:globalLF}
 \numericalFlux(\moments[1], \moments[2]) = \dfrac{1}{2} \left( \Flux_3(\moments[1]) + \Flux_3(\moments[2]) - C ( \moments[2] - 
 \moments[1]) \right).
\end{align}
The numerical viscosity constant $\viscosityconstant$ is taken as the global estimate of the 
absolute value of the largest eigenvalue of the Jacobian
$\Flux_3'$. Following \cite{Schneider2015}, the viscosity constant can be set to $\viscosityconstant = 1$, because for the moment systems used here it can be shown that the largest eigenvalue is bounded in absolute value by one\footnote{The results in \cite{Schneider2015} prove this for $\momentorder\in\{1,2\}$ but there is no general proof of this fact for arbitrary Kershaw closures yet.}.

The local Lax-Friedrichs flux could be used instead.
This requires computing the eigenvalues of the Jacobian in every
space-time cell to adjust the value of the numerical viscosity constant $\viscosityconstant$
but possibly decreases the overall diffusivity of the scheme.
However, since {high-order space-time approximations} are considered, the decrease in
diffusivity achieved by switching to the local Lax-Friedrichs flux should be
negligible.\\

The usual approach is to expand the approximate solution $\momentsprojected$ on each interval as
\begin{align}
\label{eqn:solution_form}
\left.\momentsprojected\right|_{\cell{\cellind}}(\timevar, \z) := \momentslocal{\cellind}(\timevar, \z) := \sum_{\polybasisind=0}^{\polydegree}\momentspolynomialcoefficients{\cellind}{\polybasisind}(\timevar) 
\Testfunction[\polybasisind]\left( \frac{\z - \z_\cellind}{\dzstepsize} \right),
\end{align}
where $\Testfunction[0], \Testfunction[1], \ldots ,\Testfunction[\polydegree]$ denote a 
basis for $\SpaceOfPolynomials{\spatialorder}(\referencecell)$ with respect to the standard $\Lp{2}$-scalar product on the reference cell $\referencecell = \left(-\frac12,\frac12\right)$. It is convenient to choose an orthogonal basis like the Legendre 
polynomials scaled to the interval $\referencecell$, denoted by
\begin{align}
\label{eq:LegendrePolynomialBasis}
\Testfunction[0](\referencevar) = 1, \quad \Testfunction[1](\referencevar) = 2\referencevar, \quad \Testfunction[2](\referencevar) = 
 \frac12 (12\referencevar^2-1), \: \ldots
\end{align}
With an orthogonal basis the cell means $\momentscellmean{\cellind}$ are easily available from the 
expansion coefficients $\momentspolynomialcoefficients{\cellind}{\polybasisind}$, since
\begin{align}
\momentscellmean{\cellind}(\timevar) := \frac{1}{\dzstepsize} \int_{\cell{\cellind}} \momentslocal{\cellind}(\timevar, \z)~d\z
 = \frac{1}{\dzstepsize} \sum_{\polybasisind=0}^{\polydegree} \momentspolynomialcoefficients{\cellind}{\polybasisind}(\timevar)
 \int_{\cell{\cellind}} \Testfunction[\polybasisind]\left( \frac{\z-\z_\cellind}{\dzstepsize} \right)~d\z
 = \momentspolynomialcoefficients{\cellind}{0}(\timevar).
\label{eq:Cellmean}
\end{align}
Collecting the coefficients $\momentspolynomialcoefficients{\cellind}{\polybasisind}(\timevar)$ into the $\spatialorder \times 
(\momentorder +1)$ matrix
\begin{align}
\label{eq:momentspolynomialmatrix}
\momentspolynomialmatrix{\cellind}(\timevar) = \left( \momentspolynomialcoefficients{\cellind}{0}(\timevar),\ldots , \momentspolynomialcoefficients{\cellind}{\polydegree}(\timevar)\right)^T,
\end{align}

equation \eqref{eq:dweakform1} can be written in compact form as the coupled system of ordinary differential equations
\begin{align}
\label{eq:odeDG}
\dt \momentspolynomialmatrix{\cellind} &= \DifferentialOperatorLocal(\momentspolynomialmatrix{\cellind - 1}, \momentspolynomialmatrix{\cellind}, \momentspolynomialmatrix{\cellind + 1}), \quad 
\text{for } 
\cellind \in \{1, \ldots , \ncells\} \text{ and } \timevar \in \timeint,
\end{align}
with initial condition \eqref{eq:dweakform1b} and an appropriate choice of the local differential operator $\DifferentialOperatorLocal$ \cite{Schneider2015a}.

The incorporation of boundary conditions for moment systems is non-trivial. Here, an often-used approach is taken that incorporates boundary conditions via `ghost cells'.
First assume that it is possible to smoothly extend $\distributionboundary(\timevar,\z,\SCheight)$ in $\SCheight$ to $[-1,1]$ for $\z\in\{\zL,\zR\}$ (note that while moments are defined using integrals over all $\SCheight$, the
boundary conditions in \eqref{eq:TransportEquation1DBCa}--\eqref{eq:TransportEquation1DBCb} are only defined for
$\SCheight$ corresponding to incoming data). 

Then the moment approximations in the ghost cells at $\z_0$ and $\z_{\ncells+1}$ simply take the form 
\begin{subequations}
\label{eq:BCMomentSystemSimple1D}
\begin{align}
\momentslocal{0}(\timevar, \z_{\frac12}) &:= \ints{\basis \distributionboundary(\timevar,\zL,\SCheight)},\\
\momentslocal{\ncells + 1}(\timevar, \z_{\ncells + \frac12}) &:= \ints{\basis \distributionboundary(\timevar,\zR,\SCheight)}.
\end{align}
\end{subequations}
Note, however, that the validity of this approach, due to its inconsistency with 
the original boundary conditions  \eqref{eq:TransportEquation1DBCa}--\eqref{eq:TransportEquation1DBCb}, is not entirely non-controversial, but the question of appropriate boundary conditions for
moment models is an open problem \cite{pomraning1964variational,Larsen1991,Rulko1991,Struchtrup2000,levermore2009boundary} which is not explored here. 

For Dirichlet-boundary conditions, the simplest approach is taken. The ghost-cell moments are chosen to be the constant functions 
\begin{align*}
\momentslocal{0}(\timevar, \z) &\equiv \momentslocal{0}(\timevar, \z_{\frac12}),\\
\momentslocal{\ncells + 1}(\timevar, \z) &\equiv \momentslocal{\ncells + 1}(\timevar, \z_{\ncells + \frac12}),
\end{align*}
with $\momentslocal{0}(\timevar, \z_{\frac12})$ and $\momentslocal{\ncells + 1}(\timevar, \z_{\ncells + \frac12})$ defined as in \eqref{eq:BCMomentSystemSimple1D}.\\

For periodic boundary conditions, the obvious choice is
\begin{align*}
\momentslocal{0}(\timevar, \z) &= \momentslocal{\ncells}(\timevar, \z+\zR-\zL),\quad \z\in\cell{0},\\
\momentslocal{\ncells + 1}(\timevar, \z) &=\momentslocal{1}(\timevar, \z-\zR+\zL),\quad \z\in\cell{\ncells+1}.
\end{align*}

{All that remains to obtain a high-order scheme in space and time is a suitable time integrator for \eqref{eq:odeDG}}. Such a class of integrators is given by the \emph{strong-stability-preserving} (\emph{SSP}) methods, as used for example in \cite{Zhang2010,AllHau12}. The
stages and steps of these type of methods are convex combinations of forward-Euler steps.
Since the realizable set is convex, the analysis of a forward-Euler step
then suffices to prove realizability preservation of the {full method}.\\

When possible, \emph{SSP-Runge-Kutta} (\emph{SSP-RK}) methods are used, but unfortunately they only exist up to order four \cite{Ruuth2004,Gottlieb2005}.
{For orders $\spatialorder\geq 5$} the so-called \emph{two-step Runge-Kutta} (\emph{TSRK}) \emph{SSP}
methods \cite{Ketcheson2011} as well as their generalizations, the
\emph{multi-step Runge-Kutta} (\emph{MSRK}) \emph{SSP} methods \cite{Bresten2013} can be applied.
They combine Runge-Kutta schemes with positive weights and
high-order multistep methods to achieve a total order higher than
four while maintaining the important SSP property.

See \cite{Schneider2015b} for more information about the SSP-schemes used in the actual implementation. Note that they differ from those used in \cite{Schneider2015a} where only discretizations up to third order were used, in contrast to the methods of order one to seven given in \cite{Schneider2015b}.

The rest of the methodology follows closely \cite{Schneider2015a}. Here, the standard \emph{TVBM corrected minmod limiter} proposed in
\cite{Cockburn1989a} is used.

Assuming that the major part of the spurious oscillations is 
generated in the linear part of the underlying polynomial, whose slope in
the reference cell is 
simply $\momentspolynomialcoefficients{\cellind}{1}$, a limiter can be defined as
\begin{align}
\label{eq:slopelimiterscalar}
 \slopelimiter(\momentspolynomialmatrix{\cellind-1},\momentspolynomialmatrix{\cellind},\momentspolynomialmatrix{\cellind+1}) = \left\{ 
  \begin{array}{cc}
   \left( \begin{array}{c}
           \left( \momentspolynomialcoefficients{\cellind}{0} \right)^T \\
            \minmod{\momentspolynomialcoefficients{\cellind}{1}, \momentspolynomialcoefficients{\cellind+1}{0} - \momentspolynomialcoefficients{\cellind}{0}, 
             \momentspolynomialcoefficients{\cellind}{0} - \momentspolynomialcoefficients{\cellind-1}{0}}^T \\
           (0, 0, \ldots , 0)\\
           \vdots\\
           (0, 0, \ldots , 0)
          \end{array} \right)
    &\text{ if }
     \abs{\momentspolynomialcoefficients{\cellind}{1}} \geq \TVBconstant(\dzstepsize)^2, \\     
    \momentspolynomialmatrix{\cellind} & \text{otherwise},
  \end{array} \right.
\end{align}
for the $\cellind^{\text{th}}$ cell and the case $\spatialorder \geq 3$, that is piece-wise quadratic or higher-degree polynomials, so that the final rows of zeros in the first case indicates
that the coefficients for the higher-order {spatial} basis functions {$\Testfunction[1], \ldots ,\Testfunction[\polydegree]$ are set to zero for}
each moment component.
The absolute value and the inequality are applied componentwise. The label ``scalar'' is used because the limiter is directly applied to each 
scalar component of $\momentscollection{}$.
The function $\minmod{\cdot}$ is the standard minmod function applied componentwise, defined by
\begin{align*}
\minmod{a_1,a_2,a_3} &= \begin{cases}
 \operatorname{sign}(a_1) \min\{|a_1|,|a_2|,|a_3|\} & \text{if } 
 \operatorname{sign}(a_1) = \operatorname{sign}(a_2) = 
 \operatorname{sign}(a_3), \\
0 & \text{else}.
\end{cases}
\end{align*}

The constant $\TVBconstant$ is a problem-dependent estimate of the second derivative, 
though it has to be noted that in \cite{Cockburn1989a} the authors did not find the 
solutions very sensitive to the value chosen for this parameter.

However, it has been found that applying the limiter to the components 
themselves may introduce non-physical oscillations around an otherwise monotonic 
solution \cite{Cockburn1989}.
Instead, the limiter is applied to the local characteristic fields of the 
solution. The flux Jacobian is obtained numerically using finite differences. This can be achieved cheaply since only the last equation of the flux has to be considered, all other components are trivial.

\subsection{Realizability preservation}
\label{sec:RealizabilityPreservation1D}
% % % % % % % % % % % % % % % % % % % % % % % % % % % % % % % % % % % % % % % % % % % % % % % % % % %
% % % % %-----------------------Realizability preservation----------------------- % % % % %
% % % % % % % % % % % % % % % % % % % % % % % % % % % % % % % % % % % % % % % % % % % % % % % % % % %

In order to evaluate the flux-term $\Flux_3(\momentsprojected(\timevar, \z))$ at the spatial quadrature nodes 
$\spatialQuadratureNodes{\cellind}{\spatialQuadratureIndex}$ in the $\cellind^{\text{th}}$ cell, at least $\momentslocal{\cellind}(\spatialQuadratureNodes{\cellind}{\spatialQuadratureIndex}) =: \momentspv{\cellind,\spatialQuadratureIndex} \in 
\RD{\basis}{}$ for each node is necessary\footnote{Although intuition expects $\momentslocal{\cellind}\left(\timevar,\z\right)\in\RD{\basis}{}$ for all $\z\in\cell{\cellind}$, having realizable point values only indeed suffices to preserve realizability of the updated cell means.}.

To prove \thmref{thm:MainTheorem1D} the following rather strong assumption has to be made.\\

\begin{assumption}
\label{ass:collisionopassumption}
For every $\distribution$ satisfying $\moments = \ints{\basis\distribution}$ there exists a $\collisionrealizablepart\in\RD{\basis}{}$ such that the moments of the collision operator $\collisionop$ with respect to the same angular basis $\basis$ can be written as
\begin{align}
\label{eq:collisionopassumption}
\ints{\basis \collision{\distribution}} = \collisionrealizablepart-\moments.
\end{align}
\end{assumption}

This assumption is fulfilled by the integral collision operator \eqref{eq:collisionOperatorLin1D}.

While first-order schemes {(like the Lax-Friedrichs method)} automatically preserve realizability of the cell means \cite{Schneider2015a}, higher-order schemes {($\spatialorder\geq 2$)} typically cannot guarantee this property on their own, as
has been 
observed in the context of the compressible Euler equations (which are
indeed in the hierarchy of minimum-entropy models) in \cite{Zhang2010} and for the $\MN[1]$ model in \cite{Olbrant2012}.\\

It is, however, possible to show that, when the moments at the quadrature nodes are 
realizable, the presented schemes preserve realizability of the cell means 
$\momentscellmean{\cellind}(\timevar)$ under a CFL-type condition. 
With realizable cell means available, a point-wise-realizable polynomial representation can be obtained by
applying a linear scaling limiter which pushes $\momentspv{\cellind,\spatialQuadratureIndex}$ towards the cell mean and thus into the 
realizable set for each quadrature node $\spatialQuadratureNodes{\cellind}{\spatialQuadratureIndex}$. 

Following the arguments in \cite{Zhang2010a,Zhang2011}, this limiter does not
destroy the accuracy of the scheme in case of smooth solutions if $\momentscellmean{\cellind}$
is not on the boundary of the realizable set.
This is verified numerically in \secref{sec:Convergence}. For convenience, the main result of \cite{Schneider2015a,Schneider2015b} is summarized in the following theorem. Note that, since SSP time integrators are used, it suffices to investigate forward Euler steps in time, which are then convexly combined to obtain the {designed order of time integration in the SSP scheme}.

\begin{theorem}[\cite{Schneider2015a,Schneider2015b}]
\label{thm:MainTheorem1D}
Assume that 
\begin{enumerate}[(i)]
 \item for all cells $\cellind \in \{1, 2, \ldots, \ncells\}$ it holds that
  $0 \leq \source(\timepoint{\timeind}, \z),\absorption(\timepoint{\timeind}, \z),\scattering(\timepoint{\timeind}, \z)\in\FiniteElementSpace{\spatialordersource}$, $\spatialordersource\in\N$;
 \item the cell means $\momentscellmeantime{\cellind}{\timeind}$ at time step $\timepoint{\timeind}$ are
  realizable;
 \item at the quadrature nodes of the
   $\spatialQuadratureNumber$-point Gauss-Lobatto rule on each cell $\cell{\cellind}$, $\spatialQuadratureNumber = \ceil{\frac{\spatialorder + \spatialordersource + 1}{2}}$,%
     \footnote{Where $\ceil{\cdot}$ is the ceiling function, that is, it
      returns smallest integer bigger than or equal to its argument. Since the Gauss-Lobatto rule is exact for polynomials of degree $2\spatialQuadratureNumber-3$ this choice guarantees to exactly integrate the occurring polynomials of degree $\left(\spatialorder + \spatialordersource-2\right)$.} the point-wise values of  the moment approximation $\momentsprojected$ (componentwise in $\FiniteElementSpace{\spatialorder}$) are realizable.
\end{enumerate}
Then under the CFL condition
\begin{align}
\label{eq:CFL1D}
 \dtstepsize < \min \left( \cfrac{1}{\maxcrosssection},
  \cfrac{\dzstepsize \spatialQuadratureWeightsRef{1}}{1
  + \dzstepsize \spatialQuadratureWeightsRef{1} \maxcrosssection} \right),
\end{align}
the cell means $\momentscellmeantime{\cellind}{\timeind+1}$ after one forward-Euler step are realizable, where
\begin{align}
\label{eq:sigmatmax}
\maxcrosssection := \max\limits_{\substack{\cellind\in\{1,\ldots,\ncells\}\\\spatialQuadratureIndex\in\{1,\ldots,\spatialQuadratureNumber\}}} \crosssection\left(\timepoint{\timeind},\spatialQuadratureNodes{\cellind}{\spatialQuadratureIndex}\right).
\end{align}
\end{theorem}

All that remains is to ensure that assumption (iii) in \thmref{thm:MainTheorem1D} is always fulfilled. Due to assumption (ii) and the convexity of the realizable set, this can be achieved using a linear-scaling limiter, pushing the polynomial representation towards the (realizable) cell mean. This approach has been outlined in \cite{Zhang2010,Zhang2010a,Zhang2011a} for the Euler equations and in \cite{Schneider2015a} for two classes of minimum-entropy models.

For ease of notation, time indices are dropped.

Recall the definition \eqref{eqn:solution_form} of $\momentslocal{\cellind}$, given by
\begin{align*}
\momentslocal{\cellind}(\z) = \momentscellmean{\cellind} + \sum_{\polybasisind=1}^{\polydegree}\momentspolynomialcoefficients{\cellind}{\polybasisind} 
\Testfunction[\polybasisind]\left( \frac{\z - \z_\cellind}{\dzstepsize} \right).
\end{align*}

Due to the convexity of the realizable set, if $\momentscellmean{\cellind}$ is 
realizable, then for each quadrature point there exists a $\limitervariable \in 
[0,1]$ such that
\begin{align}
\label{eq:momentslimited}
 \momentslocallimited[\limitervariable]{\cellind}(\spatialQuadratureNodes{\cellind}{\spatialQuadratureIndex}) := \momentslimitedpv[\limitervariable]{\cellind,\spatialQuadratureIndex}  := \limitervariable \momentscellmean{\cellind} + (1 - \limitervariable) \momentspv{\cellind,\spatialQuadratureIndex} 
\end{align}
is realizable.  Indeed, by inserting the definition of $\momentslocal{\cellind}(\spatialQuadratureNodes{\cellind}{\spatialQuadratureIndex})$ from 
above, the limited moment vector can be written as
\begin{align*}
\momentslimitedpv[\limitervariable]{\cellind,\spatialQuadratureIndex} = \momentscellmean{\cellind} + \left(1-\limitervariable\right)\sum_{\polybasisind=1}^{\polydegree}\momentspolynomialcoefficients{\cellind}{\polybasisind}
\Testfunction[\polybasisind]\left( \frac{\spatialQuadratureNodes{\cellind}{\spatialQuadratureIndex} - \z_\cellind}{\dzstepsize} \right),
\end{align*}
thus when limiting is necessary, the higher-order coefficients $\momentspolynomialcoefficients{\cellind}{\polybasisind}$, {$\polybasisind = 1,\ldots,\polydegree$}, 
are damped while the cell mean remains unchanged.\\

The task of the limiter is now to choose for each $\momentslocal{\cellind}$ the minimal value of $\limitervariable_\cellind\in[0,1]$ such that $\momentslocallimited[\limitervariable_\cellind]{\cellind}$ is realizable at all quadrature nodes $\spatialQuadratureNodes{\cellind}{\spatialQuadratureIndex}$. This choice is optimal in the sense that the least information of the original polynomial is lost ($\limitervariable=0$ corresponds to no limiting while $\limitervariable=1$ resembles limiting to first order).\\

\begin{remark}
For readability reasons, the dependence on the cell index $\cellind$ is dropped sometimes throughout the following examples.
\end{remark}

Having the non-linear structure of the full realizable set $\RD{\basis}{}$ in mind, computing the smallest 
$\limitervariable$ such that $\momentslimitedpv[\limitervariable]{\cellind,\spatialQuadratureIndex} \in \RD{\basis}{}$ requires some effort. 

\begin{theorem}
The solution to the limiter problem 
\begin{align*}
\min~ &\limitervariable\\
\text{s.t. }& \momentslimitedpv[\limitervariable]{\cellind,\spatialQuadratureIndex} \in \RD{\fmbasis}{}\\
&\limitervariable\in[0,1]
\end{align*}
requires to calculate the roots of two polynomials of degree at most $\momentorder$.
\end{theorem}
\begin{proof}
Assume that $\momentorder=2\hankelhalfind+1$. Define $\overline{\hankelA}$, $\hankelA$, $\overline{\hankelB}$ and $\hankelB$ to be the Hankel matrices associated with $\momentscellmean{\cellind}$ and $\momentslocal{\cellind}(\spatialQuadratureNodes{\cellind}{\spatialQuadratureIndex})$, respectively. Then the Hankel matrices associated with $\momentslimitedpv[\limitervariable]{\cellind,\spatialQuadratureIndex}$ are 
\begin{align*}
\hankelA^\limitervariable := \limitervariable\overline{\hankelA}+(1-\limitervariable)\hankelA,\\
\hankelB^\limitervariable := \limitervariable\overline{\hankelB}+(1-\limitervariable)\hankelB.
\end{align*}
By assumption (compare \lemref{thm:FullMomentRealizability}), $\overline{\hankelA}\geq \pm \overline{\hankelB}$. This implies that all eigenvalues of $\overline{\hankelA}\mp \overline{\hankelB}$ are non-negative, and therefore $\det\left(\overline{\hankelA}\mp \overline{\hankelB}\right)\geq 0$. Being on the realizability boundary corresponds to having at least one zero eigenvalue, which is equivalent to a vanishing determinant of either $\hankelA^\limitervariable-\hankelB^\limitervariable$ or $\hankelA^\limitervariable+\hankelB^\limitervariable$. Note that $\det\left(\hankelA^\limitervariable\pm\hankelB^\limitervariable\right)$ is a polynomial of degree $\momentorder$ in $\limitervariable$. Since the realizable set is convex, the maximal $\limitervariable$ in $[0,1]$ that is a root of one of the two polynomials is the optimal limiter value. 

The case $\momentorder=2\hankelhalfind$ works analogously.
\end{proof}

\begin{example}
\label{eq:M1Limiter1D}
Realizability conditions for $\momentorder = 1$ are very simple: $\momentcomp{0}\geq \pm\momentcomp{1}$. Plugging in $\momentslimitedpv[\limitervariable]{}$ from \eqref{eq:momentslimited} gives
\begin{align*}
\limitervariable\momentcompcellmean{0}+\left(1-\limitervariable\right)\momentcomp{0} \geq \pm \limitervariable\momentcompcellmean{1}\pm\left(1-\limitervariable\right)\momentcomp{1}.
\end{align*}
Solving these equations for equality (which is equivalent to finding roots of a polynomial of degree $\momentorder = 1$) results in 
\begin{align*}
\limitervariable_\pm = \frac{\momentcomp{0} \mp \momentcomp{1}}{\momentcomp{0} \mp \momentcomp{1} - \momentcompcellmean{0} \pm \momentcompcellmean{1}}.
\end{align*}
\end{example}

\begin{example}
For $\momentorder = 2$ the realizability conditions are given through the Hankel matrices 
\begin{align*}
\hankelA(0) = \momentcomp{0},~
\hankelC(1) = \momentcomp{2},~
\hankelA(1) = \begin{pmatrix}
\momentcomp{0} & \momentcomp{1}\\
\momentcomp{1} & \momentcomp{2}
\end{pmatrix}
\end{align*}
and the conditions $\hankelA(1)\geq 0$ and $\hankelA(0) \geq \hankelC(1)$. The matrices defining the limiter value $\limitervariable$ are given by 
\begin{align*}
D_1(\limitervariable) &= \hankelA^\limitervariable(0)-\hankelC^\limitervariable(1) = \limitervariable(\momentcompcellmean{0}-\momentcompcellmean{2})+(1-\limitervariable)(\momentcomp{0}-\momentcomp{2}),\\
D_2(\limitervariable) &= \hankelA^\limitervariable(1) = \limitervariable\begin{pmatrix}
\momentcompcellmean{0} & \momentcompcellmean{1}\\
\momentcompcellmean{1} & \momentcompcellmean{2}
\end{pmatrix}
+(1-\limitervariable)\begin{pmatrix}
\momentcomp{0} & \momentcomp{1}\\
\momentcomp{1} & \momentcomp{2}
\end{pmatrix}.
\end{align*}
The required polynomials are given by $p_{1,2} (\limitervariable)= \det\left(D_{1,2}(\limitervariable)\right)$, i.e.
\begin{align*}
p_1(\limitervariable) &= \limitervariable(\momentcompcellmean{0}-\momentcompcellmean{2})+(1-\limitervariable)(\momentcomp{0}-\momentcomp{2})\\
p_2(\limitervariable) &= \left( - {\momentcomp{1}}^2 + 2\, \momentcomp{1}\, \momentcompcellmean{1} - {\momentcompcellmean{1}}^2 - \momentcomp{2}\, \momentcompcellmean{0} + \momentcompcellmean{0}\, \momentcompcellmean{2} + \momentcomp{0}\, \left(\momentcomp{2} - \momentcompcellmean{2}\right)\right)\, \limitervariable^2,\\
& \quad+ \left(2\, {\momentcomp{1}}^2 - 2\, \momentcompcellmean{1}\, \momentcomp{1} + \momentcomp{2}\, \momentcompcellmean{0} - \momentcomp{0}\, \left(2\, \momentcomp{2} - \momentcompcellmean{2}\right)\right)\, \limitervariable + \left(\momentcomp{0}\, \momentcomp{2} - {\momentcomp{1}}^2\right).
\end{align*}
Let e.g. $\momentscellmean{} = \left(1,0,\frac13\right)^T$ and $\moments = \left(1,\frac45,\frac15\right)^T$. Then it follows that
\begin{align*}
p_1(\limitervariable) &= \frac45-\frac{2}{15}\limitervariable,\\
p_2(\limitervariable) &=-\frac{11}{25}+\frac{106}{75}\limitervariable-\frac{16}{25}\limitervariable^2,
\end{align*}
which have roots $\limitervariable_1 = 6$, $\limitervariable_{2+} = \frac{11}{6}$ and $\limitervariable_{2-} = \frac38$. Since $\limitervariable_1,\limitervariable_{2+} \notin [0,1]$ it follows that $\limitervariable = \limitervariable_{2-} $ and $\momentslimitedpv[\limitervariable]{} = \left(1,\frac12,\frac14\right)^T$, which indeed satisfies the second-order realizability condition $\momentcomplimitedpv{2}\momentcomplimitedpv{0}\geq \momentcomplimitedpv{1}\momentcomplimitedpv{1}$ with equality. This example is visualized in \figref{fig:K2LimiterExample}.
\end{example}

\begin{figure}[htbp]
\centering
\externaltikz{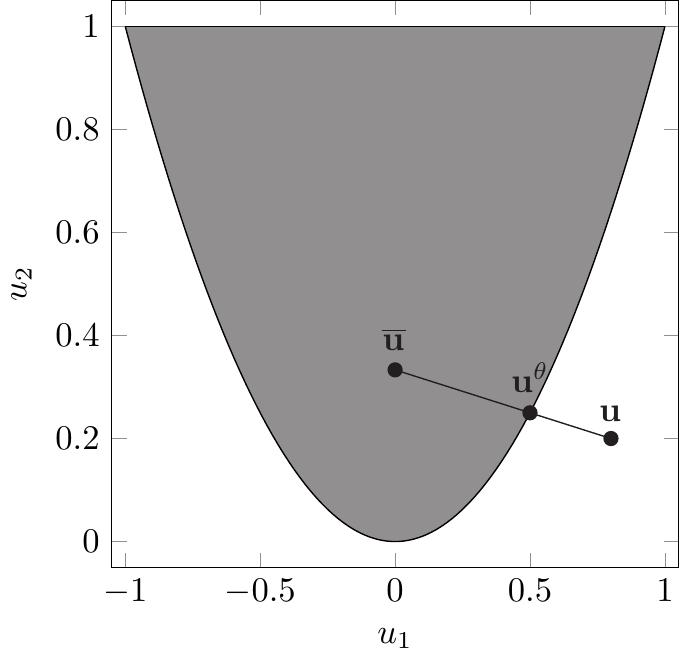}{\input{Images/K2LimiterExample}}
\caption{Limiter example for the second-order basis with $\momentcomp{0} = 1$. The realizable set is plotted in grey.}
\label{fig:K2LimiterExample}
\end{figure}

\begin{remark}
To analytically obtain the coefficients for the resulting polynomials in $\limitervariable$ is in general hard. However, this can be avoided using a simple trick. In the odd case evaluate the determinants of $\hankelA^\limitervariable\pm\hankelB^\limitervariable$ at $\momentorder+1$ distinct values of $\limitervariable$, e.g. at the $\momentorder+1$ linearly-spaced values in $[0,1]$. These points uniquely define the desired polynomial. A similar approach can be done in the even case.
\end{remark}

{\begin{remark}
It has been shown in \cite{Schneider2016} that the slope limiter \eqref{eq:slopelimiterscalar} (either evaluated in primitive or conserved variables) always has to be applied before the realizability limiter since the application of the slope limiter may destroy point-wise realizability (and thus \thmref{thm:MainTheorem1D} cannot be applied).\\
Both limiters have to be applied at every stage and step of the SSP time integrators.
\end{remark}}
\section{Numerical experiments}
This section contains numerical convergence results and some often-used benchmark problems for moment models. They serve as a reference for the efficiency of Kershaw closures with a {high number of moments} combined with high-order space-time discretizations. 

{The approximations of highest order in space and time are discretized with }$\spatialorder=7$ on a grid with $50$ cells, the medium order is represented by a $\spatialorder=4$ solution on a grid with $100$ cells and the first-order variant is calculated on a grid with $500$ cells. If not stated otherwise the TVB constant in the modified minmod limiter is set to $\TVBconstant=5$ for $\spatialorder=7$ and $\TVBconstant = 20$ for $\spatialorder = 4$ (compare \cite{Cockburn1989a,Schneider2015a}).

\subsection{Convergence results}
\label{sec:Convergence}

\subsubsection{Manufactured solution}

In general, obtaining analytical solutions for moment models is a hard task. In the case of Kershaw closures it is possible to provide a solution in some special cases. Consider the initial distribution
\begin{align*}
\distributiontzero(\z,\SCheight) = f(\z)\dirac\left(\SCheight-1\right)
\end{align*}
with some positive $f(\z)> 0$. Setting $\absorption=\scattering=0$, the analytical solution of the transport equation \eqref{eq:TransportEquation1D} is given by 
\begin{align*}
\analyticalsolution(\timevar,\z,\SCheight) = f(\z-\timevar)\dirac\left(\SCheight-1\right).
\end{align*}
On the moment level this corresponds to a linear advection with transport speed $1$ since
\begin{align*}
\analyticalmomentcomp{\basisind}(\timevar,\z) = \ints{\SCheight^\basisind \analyticalsolution(\timevar,\z,\SCheight)} = f(\z-\timevar)\qquad\text{ for all } \basisind\in\{0,\ldots,\momentorder\}.
\end{align*}
Since $\dirac\left(\SCheight-1\right)$ can be reproduced exactly by the Kershaw closures \cite{Schneider2015} the moments of the transport solution are also the moments of the Kershaw closure if $\momentorder\geq 1$. For this example the local mass is defined as $f(\z) = \sin(\z)$ on $\Domain = [-\pi,\pi]$. The final time is set to $\tf = 0.2\pi$ and periodic boundary conditions are applied.

Errors are computed in the zeroth moment of the solution $\analyticalmomentcomp{0}(\timevar, \z) := \ints{\analyticalsolution(\timevar, \z, \cdot)}$.
Then $\Lp{1}$- and $\Lp{\infty}$-errors for the zeroth moment $\momentcompprojected{0}(\timevar, \z)$
(that is, the zeroth component of a numerical solution $\momentsprojected$) are
defined as
\begin{equation}
 \LpError{1} = \int_\Domain \left| \analyticalmomentcomp{0}(\tf, \z) - \momentcompprojected{0}(\tf, \z) \right|~d\z
  \quand
\LpError{\infty} = \max_{\z \in \Domain} \left| \analyticalmomentcomp{0}(\tf, \z) - \momentcompprojected{0}(\tf, \z) \right|,
\label{eq:errors}
\end{equation}
respectively. The integral in $\LpError{1}$ is approximated using a 100-point Gauss-Lobatto quadrature rule over each spatial cell $\cell{\cellind}$, and $\LpError{\infty}$ is
approximated by taking the maximum over these quadrature nodes.
The observed convergence order $\MFSconstorder$ is defined by
\begin{equation}
 \frac{\LpError[h_1]{\MFSp}}{\LpError[h_2]{\MFSp}} = \left( \frac{\dzstepsize_1}{\dzstepsize_2} \right)^\MFSconstorder,
\label{eq:conv-order}
\end{equation}
where $\LpError[h_i]{\MFSp}$, $i \in \{1, 2\}$, $\MFSp \in \{1, \infty\}$, is the $\Lp{\MFSp}$-error $\LpError{\MFSp}$ for the
numerical solution using cell size $\dzstepsize_i$.

A convergence table for orders $\spatialorder\in\{2,4,5,6,7\}$ is presented in \tabref{tab:ConvergenceDG}.

\begin{table}[h]
\centering
\begin{tabular}{r r@{.}l c r@{.}l c r@{.}l c r@{.}l c r@{.}l c }
& \multicolumn{3}{c}{$\spatialorder = 2 $}& \multicolumn{3}{c}{$\spatialorder = 4 $}& \multicolumn{3}{c}{$\spatialorder = 5 $}& \multicolumn{3}{c}{$\spatialorder = 6 $}& \multicolumn{3}{c}{$\spatialorder = 7 $}\\
\cmidrule(r){2-4} \cmidrule(r){5-7} \cmidrule(r){8-10} \cmidrule(r){11-13} \cmidrule(r){14-16} 
$\ncells $ & \multicolumn{2}{c}{$\LpError{1}$} & $\MFSconstorder$ & \multicolumn{2}{c}{$\LpError{1}$} & $\MFSconstorder$ & \multicolumn{2}{c}{$\LpError{1}$} & $\MFSconstorder$ & \multicolumn{2}{c}{$\LpError{1}$} & $\MFSconstorder$ & \multicolumn{2}{c}{$\LpError{1}$} & $\MFSconstorder$\\ \midrule 
 10 & 7 & 721e-02 & ---& 1 & 418e-04 & ---& 4 & 449e-06 & ---& 1 & 235e-07 & ---& 2 & 595e-09 & ---\\
20 & 2 & 168e-02 & 1.8& 9 & 004e-06 & 4.0& 1 & 430e-07 & 5.0& 1 & 898e-09 & 6.0& 2 & 107e-11 & 6.9\\
40 & 1 & 017e-02 & 1.1& 5 & 788e-07 & 4.0& 4 & 465e-09 & 5.0& 2 & 951e-11 & 6.0& 1 & 715e-13 & 6.9\\
80 & 2 & 580e-03 & 2.0& 3 & 667e-08 & 4.0& 1 & 401e-10 & 5.0& 4 & 629e-13 & 6.0& 5 & 961e-14 & 1.5\\
160 & 6 & 467e-04 & 2.0& 2 & 296e-09 & 4.0& 4 & 408e-12 & 5.0& 3 & 028e-14 & 3.9& 1 & 159e-13 & -1.0\\
\cmidrule(r){2-4} \cmidrule(r){5-7} \cmidrule(r){8-10} \cmidrule(r){11-13} \cmidrule(r){14-16} 
$\ncells $ & \multicolumn{2}{c}{$\LpError{\infty}$} & $\MFSconstorder$ & \multicolumn{2}{c}{$\LpError{\infty}$} & $\MFSconstorder$ & \multicolumn{2}{c}{$\LpError{\infty}$} & $\MFSconstorder$ & \multicolumn{2}{c}{$\LpError{\infty}$} & $\MFSconstorder$ & \multicolumn{2}{c}{$\LpError{\infty}$} & $\MFSconstorder$\\ \midrule 
 10 & 5 & 978e-02 & ---& 1 & 864e-04 & ---& 6 & 520e-06 & ---& 1 & 750e-07 & ---& 4 & 422e-09 & ---\\
20 & 1 & 592e-02 & 1.9& 1 & 155e-05 & 4.0& 2 & 106e-07 & 5.0& 2 & 836e-09 & 5.9& 3 & 451e-11 & 7.0\\
40 & 4 & 737e-03 & 1.7& 7 & 225e-07 & 4.0& 6 & 603e-09 & 5.0& 4 & 470e-11 & 6.0& 3 & 051e-13 & 6.8\\
80 & 1 & 180e-03 & 2.0& 4 & 517e-08 & 4.0& 2 & 046e-10 & 5.0& 7 & 745e-13 & 5.9& 5 & 906e-14 & 2.4\\
160 & 2 & 930e-04 & 2.0& 2 & 841e-09 & 4.0& 6 & 321e-12 & 5.0& 9 & 459e-14 & 3.0& 5 & 729e-14 & 0.0\\
\end{tabular}
\caption{$\Lp{1}$- and $\Lp{\infty}$-errors and observed convergence order $\MFSconstorder$
for the $\KN[1]$ analytical solution.}
\label{tab:ConvergenceDG}
\end{table}

It can be observed that the expected convergence rates are achieved both in $\Lp{1}$-
and $\Lp{\infty}$-errors. Note that the high-order methods ($\spatialorder\geq 5$) stop converging at an $\Lp{\infty}$-error of magnitude $10^{-14}$. This is also visible in \figref{fig:ConvergenceDGConvergence}, where orders up to $\spatialorder=7$ are plotted together with their corresponding optimal convergence rates (black dashed line).

In \figref{fig:ConvergenceDGEfficiency} the $\Lp{\infty}$-error versus the
computation time (computed on a Intel Core i$7$ CPU with $2.8$ GHz on a single thread) is shown.
Here it is clearly visible that {efficiency rises with increasing order $\spatialorder$}.

\begin{figure}[htbp]
\centering
\settikzlabel{fig:ConvergenceDGConvergence}
\settikzlabel{fig:ConvergenceDGEfficiency}
\externaltikz{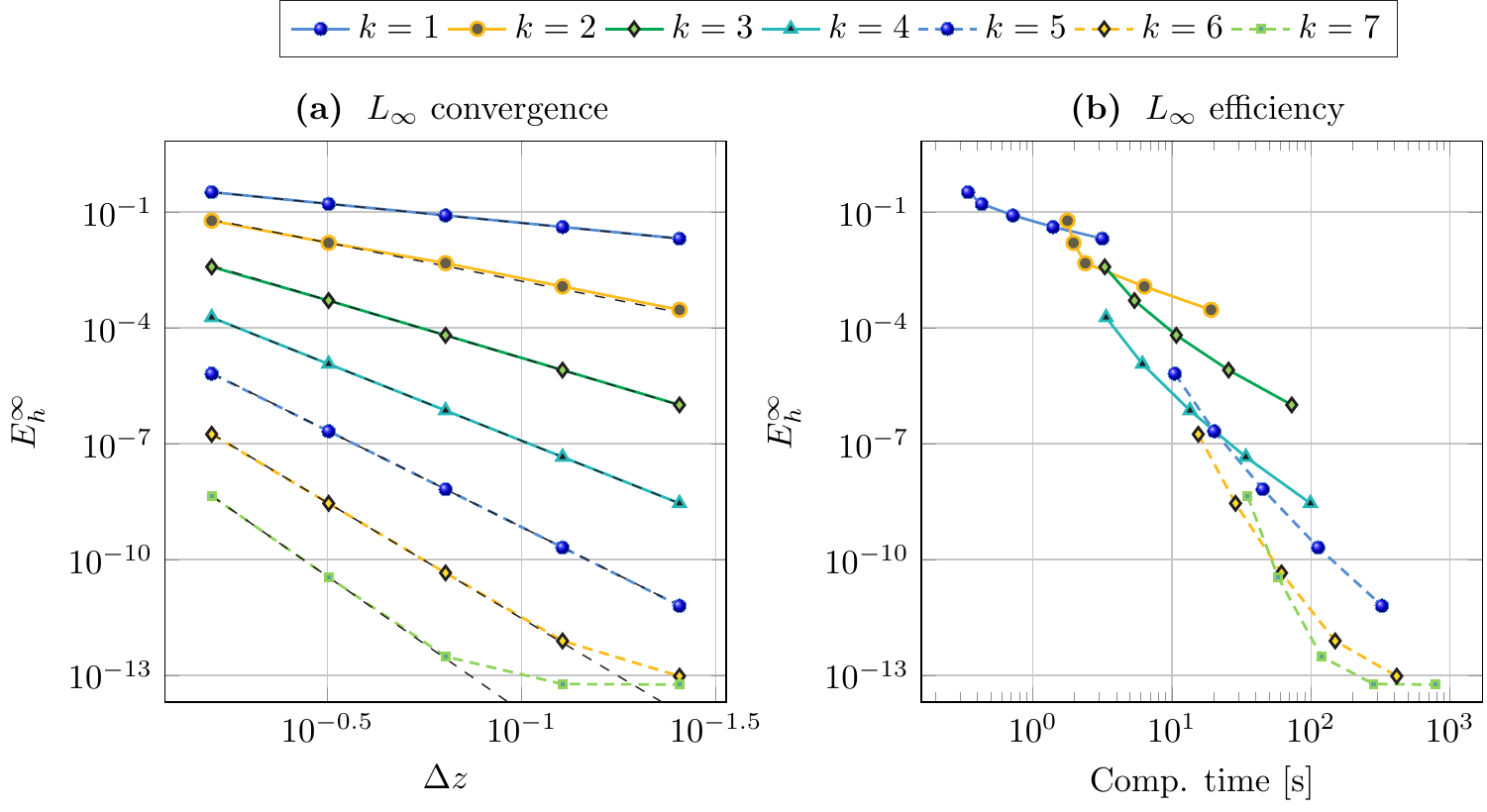}{\relinput{Images/Convergence}}
\caption{Convergence results in $\Lp{\infty}$-norm for different spatial orders and $\momentorder=1$. Black dashed lines represent the expected convergence.}
\label{fig:ConvergenceDG}
\end{figure}

Similar results can be observed using $\momentorder>1$ (tested for $\momentorder\in\{2,3\}$) but in this extreme case (the moments are always on the first-order realizability boundary) the closure procedure is more prone to numerical errors reducing the overall accuracy.

\subsubsection{Investigation of the realizability limiter}
{
Despite choosing a manufactured solution close to the boundary of
realizability, the realizability limiter was not consistently active in the previous
simulations. Therefore in this section an artificially-defined curve of moment vectors
in space is given, and reconstructed in the finite-element space $\FiniteElementSpace{2}$ of
discontinuous quadratic polynomials. Finally, the realizability limiter is used to move the reconstruction back into the set of numerically-realizable moments $\RD{\basis}{}$.
The convergence of this limited reconstruction is measured as before. This test case has been used before in \cite{Schneider2015a} for quadrature-based minimum-entropy models.
}
Using the Dirac-delta distribution $\dirac = \dirac(\SCheight)$, two
moment vectors
\begin{align*}
 \moments[0] &:= (1 - \MFSgamma) \ints{\fmbasis \dirac(\SCheight - 1)} + \MFSgamma \isotropicmoment
  = (1 - \MFSgamma) \fmbasis(1) + \MFSgamma \isotropicmoment, \\
 \moments[1] &:= 10^{-8} \left( (1 - \MFSgamma) \ints{\fmbasis \dirac(\SCheight + 1)}
  + \MFSgamma \isotropicmoment \right)
  = 10^{-8} \left( (1 - \MFSgamma) \fmbasis(-1) + \MFSgamma \isotropicmoment \right)
\end{align*}
{are chosen, which can lie arbitrarily close to the boundary of the numerically-realizable set. The parameter $\MFSgamma \in [0, 1]$ controls the distance to the boundary.
For $\momentorder > 1$, both $\moments[0]$ and $\moments[1]$ lie on the boundary of the realizable set
when $\MFSgamma = 0$. By definition, $\moments[0]$ and $\moments[1]$ (and any convex
combination thereof) are in $\RD{\basis}{}$ for $\MFSgamma \in [0, 1]$, and
so a realizable curve of moments in space is defined by taking convex
combinations of $\moments[0]$ and $\moments[1]$, i.e. }
\begin{equation}
 \moments(\z) := (1 - \convexscalar(\z)) \moments[0] + \convexscalar(\z) \moments[1], \quad \z \in [-1, 1],
\label{eq:u-limit-test}
\end{equation}
{where $\convexscalar(\z) \in [0, 1]$ is chosen to be}
$$
\convexscalar(\z) := \frac{\cos(\pi \z) + 1}{2}, \quad \z \in [-1, 1].
$$

{To perform the convergence test, $\moments(\z)$ is projected onto $\FiniteElementSpace{2}$ and $\FiniteElementSpace{3}$ for
increasing numbers of cells $\ncells$. Then the realizability
limiter is applied to ensure a realizable polynomial representation.
Errors and observed convergence orders are computed as in \eqref{eq:errors} and
\eqref{eq:conv-order}, respectively.
}
Numerical experiments show that taking $\MFSgamma \in [0, 10^{-2}]$ places the moment curve
$\moments(\z)$ close enough to the boundary of realizability that the realizability
limiter is active for every considered number of cells.

\begin{table}
\begin{tabular}{r r@{.}l c r@{.}l c r@{.}l r@{.}l c r@{.}l c r@{.}l}

 & \multicolumn{8}{c}{$\spatialorder = 2$} & \multicolumn{8}{c}{$\spatialorder = 3$}\\
\cmidrule(r){2-9} \cmidrule(r){10-17}
$\ncells$ & \multicolumn{2}{c}{$\LpError{1}$} & $\MFSconstorder$
 & \multicolumn{2}{c}{$\LpError{\infty}$} & $\MFSconstorder$
 & \multicolumn{2}{c}{$\theta_{\max}$}
 & \multicolumn{2}{c}{$\LpError{1}$} & $\MFSconstorder$
 & \multicolumn{2}{c}{$\LpError{\infty}$} & $\MFSconstorder$
 & \multicolumn{2}{c}{$\limitervariable_{\max}$} \\ \midrule
 
 10 & 9 & 226e-03 & --- & 3 & 095e-02 & --- & 3 & 287e-01& 4 & 895e-04 & --- & 1 & 025e-03 & --- & 8 & 388e-03\\
20 & 2 & 192e-03 & 2.1 & 8 & 096e-03 & 1.9 & 3 & 320e-01& 5 & 545e-05 & 3.1 & 1 & 276e-04 & 3.0 & 2 & 108e-03\\
40 & 5 & 255e-04 & 2.1 & 2 & 047e-03 & 2.0 & 3 & 329e-01& 6 & 745e-06 & 3.0 & 1 & 608e-05 & 3.0 & 5 & 276e-04\\
80 & 1 & 286e-04 & 2.0 & 5 & 131e-04 & 2.0 & 3 & 331e-01& 8 & 373e-07 & 3.0 & 2 & 014e-06 & 3.0 & 1 & 319e-04\\
160 & 3 & 182e-05 & 2.0 & 1 & 284e-04 & 2.0 & 3 & 331e-01& 1 & 045e-07 & 3.0 & 2 & 519e-07 & 3.0 & 3 & 292e-05
\end{tabular}
\caption{$\Lp{1}$- and $\Lp{\infty}$-errors and observed convergence order $\MFSconstorder$
for the zeroth moment of the realizability-limited, piece-wise linear and
quadratic reconstructions of $U(\z)$ from \eqref{eq:u-limit-test} with
$\MFSgamma = 10^{-3}$ and $\momentorder=3$.}
\label{tab:limiter-test-10}
\end{table}

{In \tabref{tab:limiter-test-10} convergence rates are shown for $\MFSgamma = 10^{-3}$ and the $\KN[3]$ model.
These results show the designed convergence order.
In this table the column $\limitervariable_{\max}$ is included, which gives the
maximum value of $\limitervariable$ from the realizability limiter over all spatial
cells.
The non-zero $\limitervariable_{\max}$ in each row indicates that the
realizability limiter is active for every reconstruction.
Similar results can be observed for every moment component. Note that for $\spatialorder\geq 4$ the realizability limiter is no longer active since the approximation quality of the reconstruction is already too good.
}
\subsection{Plane source}
\label{sec:Planesource}
In this test case an isotropic distribution with all mass concentrated in the middle of an infinite domain $\z \in
(-\infty, \infty)$ is defined as initial condition, i.e.
\begin{align*}
 \distributiontzero(\z, \SCheight) = \distributionvacuum + \dirac(\z),
\end{align*}
where the small parameter $\distributionvacuum = 0.5 \times 10^{-8}$ is used to
approximate a vacuum.
In practice, a bounded domain must be used which is large
enough that the boundary should have only negligible effects on the
solution. For the final time $\tf = 1$, the domain is set to $\Domain = [-1.2, 1.2]$ (recall that for all presented models the maximal speed of propagation is bounded in absolute value by one).

At the boundary the vacuum approximation
\begin{align*}
 \distributionboundary(\timevar,\zL,\SCheight) \equiv \distributionvacuum \quand
 \distributionboundary(\timevar,\zR,\SCheight) \equiv \distributionvacuum
\end{align*}
 is used again. Furthermore, the physical coefficients are set to $\scattering \equiv 1$, $\absorption \equiv 0$ and $\source \equiv 0$.

In contrast to \cite{Schneider2015a} a smoothed version of the Dirac is used, similar to \cite{Seibold2014}, given by
\begin{align*}
 \distributiontzero(\z, \SCheight) = \distributionvacuum + \frac{1}{2\sqrt{\pi\sigma}}\exp\left(-\frac{\z^2}{4\sigma}\right),
\end{align*}
with $\sigma = 3.2\cdot 10^{-4}$. To avoid a flattening of the otherwise smooth solution due to the minmod limiter the TVB constant is chosen to be $\TVBconstant=\infty$, completely disabling the slope limiter (but not the realizability limiter).

Some solutions at the final time are shown in \figref{fig:Planesource}, calculated for different spatial orders and resolutions.

\begin{figure}[htbp]
\centering
\externaltikz{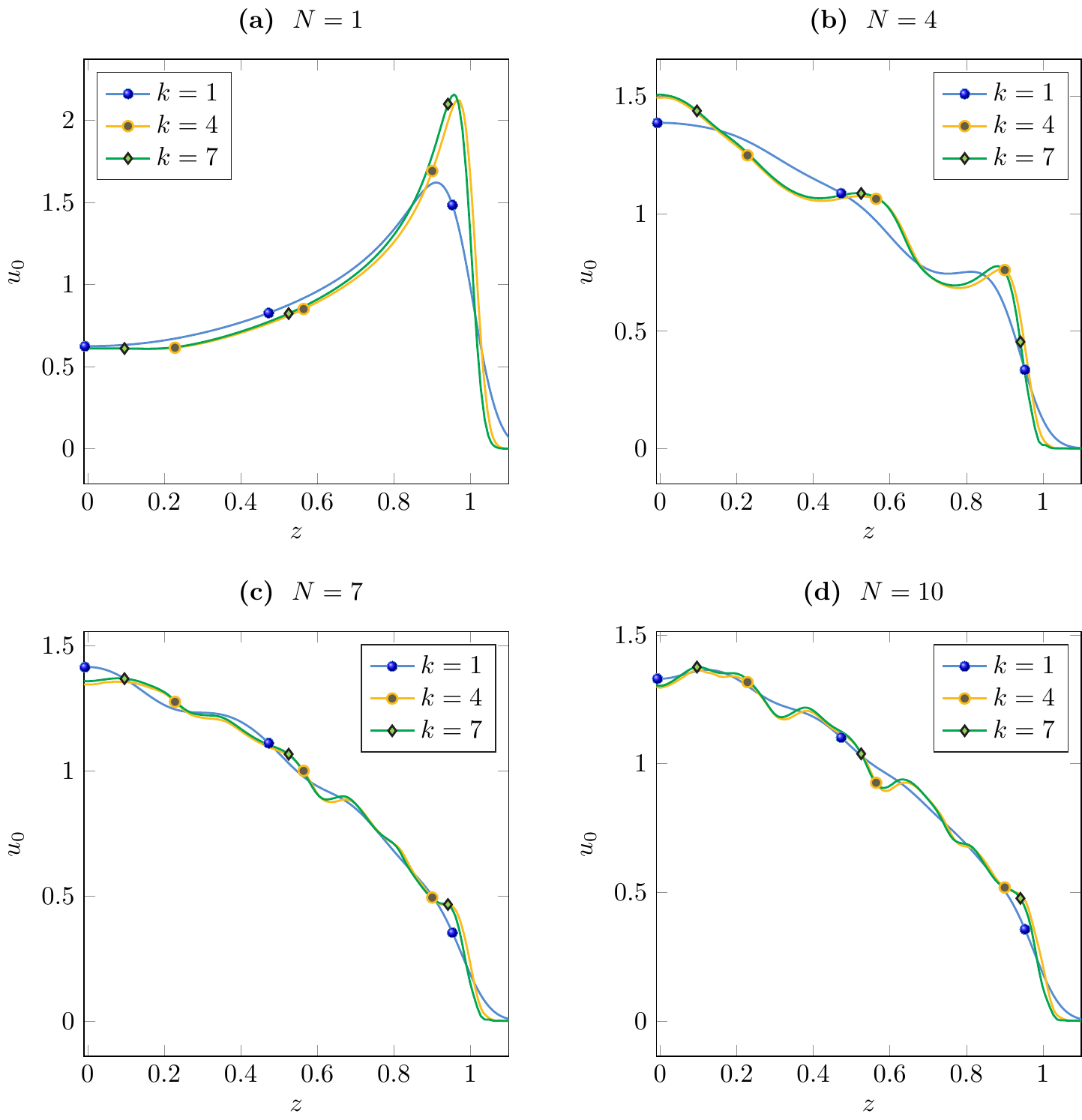}{\relinput{Images/PlanesourceIsotropicCutsFullMoments}}
\caption{Local particle density $\momentcomp{0}$ in the plane-source test case for different spatial and moment orders.}
\label{fig:Planesource}
\end{figure}

It is visible that despite its much higher resolution ($500$ cells) the first-order solution is a lot more diffusive than the higher-order results ($\spatialorder=4$ with $100$ and $\spatialorder=7$ with $50$ cells). The medium- and high-order solution largely agree though the fourth-order one appears to be slightly more diffusive. This is the case for all presented moment orders.

The activity of the realizability limiter during the simulation is presented in \figref{fig:PlanesourceLimiter}. The value of the limiter variable $\limitervariable$ is plotted in a $\z-\timevar$ diagram showing that the limiter is most active along the shock front. This is consistent with the results in \cite{Schneider2015a} where a similar test has been done for minimum-entropy models. Similarly, increasing the moment order increases the limiter activity by affecting more cells and having higher values in total.

\begin{figure}[htbp]
\centering
\externaltikz{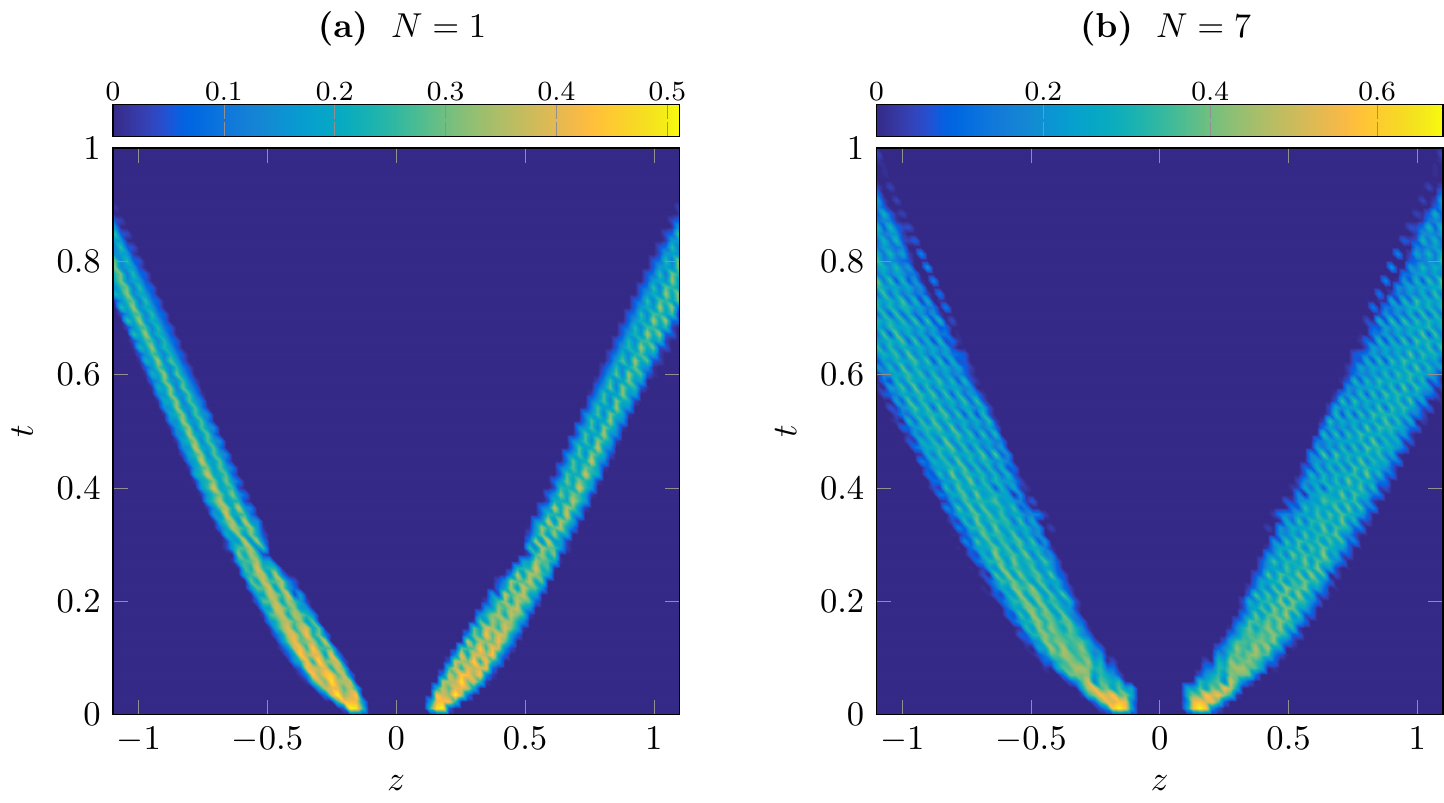}{\relinput{Images/PlanesourceLimiter}}
\caption{Realizability-limiter value $\limitervariable$ depending on $\z$ and $\timevar$ in the plane-source test for $\spatialorder = 4$.}
\label{fig:PlanesourceLimiter}
\end{figure}

\subsection{Source beam}
\label{sec:SourceBeam}
Finally, a discontinuous version of the source-beam problem from
\cite{Hauck2013} is presented.
The spatial domain is $\Domain = [0,3]$, and
\begin{gather*}
 \absorption(\z) = \begin{cases}
   1 & \text{ if } \z\leq 2,\\
   0 & \text{ else},
  \end{cases} \quad
 \scattering(\z) = \begin{cases}
   0 & \text{ if } \z\leq 1,\\
   2 & \text{ if } 1<\z\leq 2,\\
   10 & \text{ else}
  \end{cases} \quad
 \source(\z) = \begin{cases}
   1 & \text{ if } 1\leq \z\leq 1.5,\\
   0 & \text{ else},
  \end{cases}
\end{gather*}
with initial and boundary conditions
\begin{gather*}
 \distributiontzero(\z, \SCheight) \equiv \distributionvacuum, \\
 \distributionboundary(\timevar,\zL,\SCheight) = \cfrac{e^{-10^5(\SCheight-1)^2}}{\ints{e^{-10^5(\SCheight-1)^2}}}
 \quand
 \distributionboundary(\timevar,\zR,\SCheight) \equiv \distributionvacuum.
\end{gather*}
The final time is $\tf = 2.5$ and the same vacuum approximation $\distributionvacuum$ as in the plane-source problem is used.

Some solutions at the final time are shown in \figref{fig:SourceBeam}, calculated for different spatial orders and resolutions.
\begin{figure}[htbp]
\centering
\externaltikz{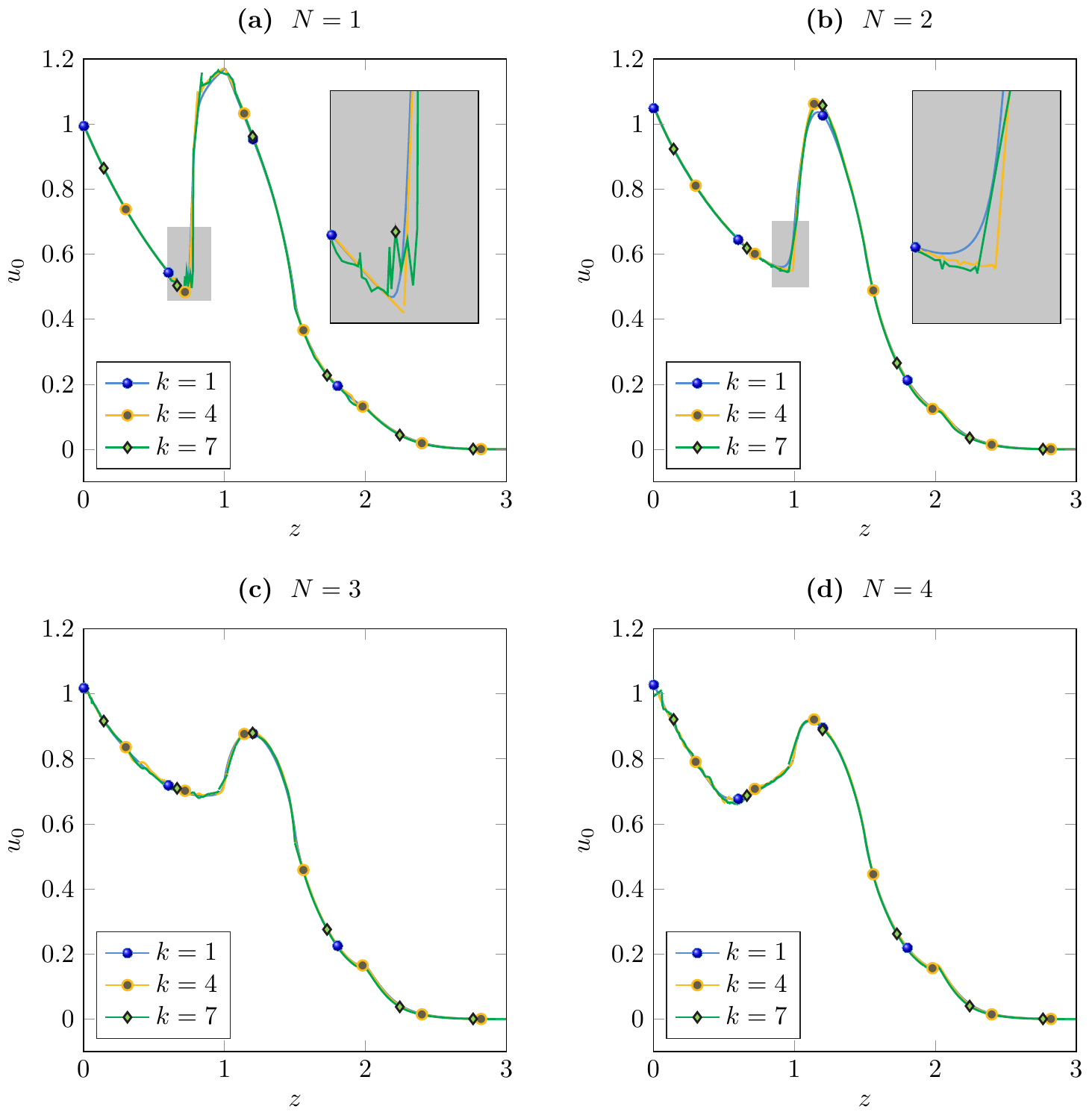}{\relinput{Images/SourceBeamIsotropicCutsFullMoments}}
\caption{Local particle density $\momentcomp{0}$ in the source-beam test case for different spatial and moment orders with zoom-ins.}
\label{fig:SourceBeam}
\end{figure}
In this non-smooth test case the benefit of high order {in space and time }is slightly diminished close to discontinuities. Even more, the seventh-order solution oscillates strongly close to the shock in the $\KN[1]$ solution. This is due to the modified minmod limiter, which is not capable to deal with such high-degree polynomials. Still, the fourth-order solution, calculated on a finer grid, is less diffusive than the first-order result. This is demonstrated in the close-up (grey box). The smoother the solution (which corresponds to increasing moment order $\momentorder$) the less oscillating the seventh-order solution. Furthermore, the different spatial approximations approach each other.
\section{Conclusions and outlook}
In this paper the necessary generalizations of the realizability-preserving discontinuous-Galerkin scheme presented in \cite{Schneider2015a} for full-moment models were derived and applied to the class of Kershaw closures. These models provide a huge gain in efficiency compared to the state-of-the-art minimum-entropy models, since they can be closed (in principle) analytically using the available realizability theory. Using high-order approximations in space and time allowed to further increase this efficiency as demonstrated in a numerical convergence test and multiple benchmark problems.

Future work will have to investigate how to adapt this scheme for different scattering operators like the slightly more complicated (in terms of realizability preservation) Laplace-Beltrami operator. Furthermore, implicit-explicit schemes should be taken into account removing the drawback that the resulting CFL condition depends on the physical parameters $\scattering$ and $\absorption$. Additionally, more sophisticated slope-limiters have to be implemented to further reduce the oscillations due to the minmod limiter.

Finally, the concepts have to be lifted to higher dimensions. While fully three-dimensional first-order variants of Kershaw closures exist \cite{Ker76,Schneider2015c}, no higher-order models or a completely closed theory is available. With this, generalizing the presented scheme is in principle possible and it can be expected that similar efficiency results hold true.
\appendix
\section{Nomenclature}
\begin{center}
{\begin{tabular}{c c c}
Symbol & Use & First occurrence\\
\cmidrule(r){1-1} \cmidrule(lr){2-2} \cmidrule(l){3-3}
$\momentcomp{\basisind}$ & $\basisind$-th scalar moment & \eqref{eq:moments}\\
$\moments$ & Moment vector, either in $\R^{\momentorder+1}$ or the solution of \eqref{eq:MomentSystemUnclosed1D} & \eqref{eq:moments}\\
$\normalizedmoments$ & Normalized moment vector, in $\R^\momentorder$ & \eqref{eq:NormalizedMoments}\\
$\momentsprojected$ & Discretized solution of \eqref{eq:MomentSystemUnclosed1D}, a vector of piecewise polynomials & \eqref{eq:dweakform1}\\
$\momentslocal{\cellind}$ & Restriction of $\momentsprojected$ to the $\cellind$-th cell, a vector of polynomials & \eqref{eqn:solution_form}\\
$\momentspolynomialcoefficients{\cellind}{\polybasisind}$ & $\polybasisind$-th coefficient vector (in $\R^{\momentorder+1}$) of the polynomial $\momentslocal{\cellind}$ wrt. the Legendre basis & \eqref{eqn:solution_form}\\
$\momentscellmean{\cellind}$ & Cell mean of the $\cellind$-th vector of polynomials & \eqref{eq:Cellmean} \\
$\momentspolynomialmatrix{\cellind}$ & Collection of all coefficient vectors $\momentspolynomialcoefficients{\cellind}{\polybasisind}$ in the $\cellind$-th cell & \eqref{eq:momentspolynomialmatrix}\\
$\momentspv{\cellind,\spatialQuadratureIndex}$& Evaluation of $\momentslocal{\cellind}$ at the quadrature node $\spatialQuadratureNodes{\cellind}{\spatialQuadratureIndex}$&\secref{sec:RealizabilityPreservation1D}\\
$\momentslocallimited{\cellind}$ & Realizability-limited version of $\momentslocal{\cellind}$ & \eqref{eq:momentslimited} \\
$\momentslimitedpv[\limitervariable]{\cellind,\spatialQuadratureIndex}$ & Evaluation of $\momentslocallimited{\cellind}$ at the quadrature node $\spatialQuadratureNodes{\cellind}{\spatialQuadratureIndex}$ & \eqref{eq:momentslimited}
\end{tabular}}
\end{center}

% Bibliography
%%%%%%%%%%%%%%
\bibliographystyle{siam}
\bibliography{bibliography}

\end{document}